\documentclass[a4paper,12pt]{article}
\usepackage[francais]{babel}

\usepackage{amsthm, amsfonts,mathrsfs,amssymb}
\usepackage[latin1]{inputenc}
\usepackage[T1]{fontenc}
\usepackage{lmodern}
\usepackage{amsfonts}
\usepackage{amsmath}

\newtheorem{theo}{Th\'eor\`eme}
\newtheorem{lemme}{Lemme}
\newtheorem{proposition}{Proposition}
\newtheorem{corollaire}{Corollaire}
\newtheorem{remarque}{Remarque}
\newtheorem{definition}{D\'efinition}

\numberwithin{theo}{section} \numberwithin{corollaire}{section}
\numberwithin{proposition}{section} \numberwithin{lemme}{section}
\numberwithin{equation}{section} \numberwithin{remarque}{section}
\numberwithin{definition}{section}

\newcommand\R{{\mathbb{R}}}
\newcommand\Z{{\mathbb{Z}}}

\makeatletter
 \def\ps@notecras{%

  \let\@mkboth\@gobbletwo
\let\sectionmark\@gobble
\let\subsectionmark\@gobble}
\makeatother

\date{}

\title{Existence et unicit\'e globale pour le syst\`eme de Navier-Stokes
axisym\'etrique anisotrope}
\author{Hammadi Abidi \footnote{D\'epartement de Math\'ematiques
Facult\'e des Sciences de Tunis
Campus universitaire 2092 Tunis, Tunisia.
habidi@univ-evry.fr}
 \; et\, Marius Paicu \footnote{Laboratoire de Math\'ematique
Universit\'e Paris Sud
B\^atiment 425,
91405 Orsay France
marius.paicu@math.u-psud.fr}}
\begin{document}
\maketitle

\noindent{\bf Abstract~:}
{\it We study in this paper the axisymmetric $3$-D Navier-Stokes system where the horizontal viscosity is zero. We prove the existence of a unique global
solution to the system with initial data of Yudovitch type.}
 \vskip 0.5cm
\noindent{\bf R\'esum\'e~:}
{\it Nous \'etudions dans ce papier le
syst\`eme de Navier-Stokes $3$-D axisym\'etrique avec viscosit\'e horizontale nulle. Nous allons prouver que le
syst\`eme est globalement bien pos\'e pour des donn\'ees de type Yudovitch.}
\vskip 0.6cm


\section{Introduction}
L'\'ecoulement tridimensionnel d'un fluide homog\`ene visqueux incompressible est r\'egi par les \'equations de Navier-Stokes que 
nous rappelons ici~:

$$
{\rm(NS)}\;\left\{
\begin{array}{rl}
&\hspace{-0,5cm}\partial_t u +(u\cdot \nabla) u
-\nu_h(\partial_{x}^2+\partial_{y}^2) u
-\nu_v\partial_{z}^2u= -\nabla p \\
&\hspace{-0,5cm}{\mathop{\rm div}}\,u=0 \\
&\hspace{-0,5cm}u_{|t=0}=u_0.
\end{array}
\right.
$$
Ci-dessus $\nu_h$ (resp. $\nu_v$) repr\'esente la viscosit\'e 
horizontale (resp. verticale), la vitesse $u$ est un
champ de vecteurs inconnu d\'ependant du temps $t$
et de la variable d'espace $x\in\R^3$ et $\nabla p$ correspond au gradient de la pression
et peut \^etre interpr\'et\'e comme le multiplicateur de Lagrange 
associ\'e \`a la contrainte d'incompressibilit\'e $\mathop{\rm div}\,u=0.$

Dans le cas ou les coefficients de viscosit\'e $\nu_h$ et $\nu_v$ sont strictement positives, on sait que le syst\`eme $(\rm{NS})$ admet une solution globale dans l'espace d'\'energie $L^2$ d'apr\`es les travaux de J. Leray \cite{L}. Ensuite dans les ann\'ees soixante H. Fujita et  T. Kato \cite{FK} ont d\'emontr\'e, par des techniques de semi-groupe que 
$(\rm{NS})$ est localement bien pos\'e pour des donn\'ees initiales 
dans l'espace de Sobolev homog\`ene $\dot H^{\frac{1}{2}}.$
L'existence globale est \'etablie pour des donn\'ees petites devant 
$\inf\{\nu_h,\nu_v\}.$ D'autres r\'esultats semblables ont \'et\'e prouv\'es dans des espaces fonctionnels qui sont tous invariants par changement d'\'echelle de l'\'equation consid\'er\'ee (voir par exemple \cite{CMP} et \cite{KT}).

Dans le cas ou $\nu_h>0$ et $\nu_v=0$ le syst\`eme $(\rm{NS_h})$ a \'et\'e \'etudi\'ee pour la premi\`ere fois par J.-Y. Chemin et al. \cite{Chemin}. Plus exactement ont
d\'emontr\'e l'existence locale en temps d'une solution, lorsque la 
donn\'ee initiale est dans l'espace de Sobolev anisotrope 
$H^{0,{1\over2}+},$ avec 
$H^{0,s}=\big\{u\in L^2\;\big|\;
(\int_{\R^2}\|u(x,y,\cdot)\|^2_{H^s(\R)}dxdy)^{1\over2}<\infty\big\}.$
L'existence globale est \'etablie pour des donn\'ees petites devant la viscosit\'e $\nu_h.$ Par contre l'unicit\'e a \'et\'e prouv\'e pour des 
donn\'ees dans $H^{0,{3\over2}+}.$
Notons que l'unicit\'e dans le cas o\`u la donn\'ee 
$u_0\in H^{0,{1\over2}+}$ a \'et\'e obtenue par D. Iftimie \cite{Iftimie}. Ensuite M. Paicu \cite{Paicu}, a d\'emontr\'e que le syst\`eme 
$(\rm{NS_h})$ est localement bien pos\'e dans l'espace de Besov anisotrope \hbox{${\mathscr{B}}^{0,{1\over2}}=\big\{u\in{\mathcal{S}}'\big|
\displaystyle\sum_{q\in\Z}(\int_{2^{q-1}\le|z|\le2^q}|z|\|{\mathcal{F}}u(\cdot,\cdot,z)\|_{L^2(\R^2)}^2dz)^{1\over2}<\infty\big\},$} l'existence globale a \'et\'e prouv\'e pour des donn\'ees petites devant $\nu_h.$ 
R\'ecemment J.-Y. Chemin et P. Zhang \cite{Ping} ont  obtenu un
r\'esultat similaire en travaillant dans un espace de Besov anisotrope
d'indice n\'egatif.

Dans la suite, on suppose que le fluide est uniquement verticalement visqueux, c'est-\`a-dire, que $\nu_h=0$ et $\nu_v>0.$
Dans cette partie on ne s'int\'eressera pas \`a la d\'ependance par rapport \`a la viscosit\'e $\nu_v$ des quantit\'es \`a mesurer, et l'on supposera donc pour simplifier que $\nu_v=1.$
Dans ce cas le syst\`eme devient~:
$$
{\rm(NS_v)}\;\left\{
\begin{array}{rl}
&\hspace{-0,5cm}\partial_t u +(u\cdot \nabla) u-\partial_{z}^2u
=-\nabla p \\
&\hspace{-0,5cm}{\mathop{\rm div}}\,u=0 \\
&\hspace{-0,5cm}u_{|t=0}=u_0.
\end{array}
\right.
$$
Rappelons que dans le cas ou $\nu_h>0$ et $\nu_v=0,$ la condition d'incompressibilit\'e, c'est-\`a-dire, $\partial_xu^1+\partial_yu^2+\partial_zu^3=0,$ a permis aux auteurs de prouver un effet r\'egularisant pour la troisi\`eme composante $u^3$ \`a partir du laplacien horizontal. Par contre dans notre cas on a  un seul effet r\'egularisant qui rend l'\'etude du syst\`eme tr\`es difficile. Pour cela on s'int\'eresse \`a des solutions particuli\`eres, plus exactement des solutions axisym\'etriques, puisque dans se cas, on a
${\mathop{\rm div}}\,u=\partial_ru^r+{u^r\over r}+\partial_zu^z=0.$
Avant de donner plus de d\'etails, il convient de pr\'eciser  ce que nous entendons par donn\'ees et solutions axisym\'etriques.
\begin{definition}
On dit qu'un  champ de vecteurs $u$ est axisym\'etrique si et seulement si il poss\`ede une sym\'etrie cylindrique de r\'eflexion, c'est-\`a-dire,
$$
u=u^r(r,z)e_r+u^z(r,z)e_z
$$
o\`u
$\big(e_r, e_{\theta} , e_z\big)$ est la base cylindrique.\\
Une fonction scalaire est dite axisym\'etrique si elle ne d\'epend pas de la variable angulaire $\theta.$
\end{definition}
Le syst\`eme de Navier-Stokes classique 
(dans le cas $\nu_h=\nu_v>0$) a d\'ej\`a  \'et\'e \'etudi\'e  par plusieurs auteurs, le premiers r\'esultats \'etant dues \`a  M. Ukhovskii et V. Youdovitch \cite{UY} et O. A. Ladyzhenskaya \cite{LA}. 
 
Dans ce cas la vorticit\'e de $u$ que est d\'efinie par 
$\omega:=\nabla\times u,$ admet dans le rep\`ere cylindrique 
une seule composante port\'ee par $e_{\theta}$:
$$
\omega=\omega^{\theta} e_{\theta}\hspace{0,5cm}\mbox{avec}
\hspace{0,5cm}\omega^{\theta}=\partial_zu^r-\partial_ru^z
$$
et qui v\'erifie l'\'equation suivante:
$$
\partial_t \omega +(u^r\partial_r+ u^z\partial_z)\omega - \frac{u^r}{r}\omega-\partial^2_z\omega=0,
$$
et par suite $\omega/r$ v\'erifie l'\'equation de trasport-diffusion:
$$
\partial_t {\omega\over r} +(u^r\partial_r+ u^z\partial_z){\omega\over r}
- \partial^2_z{\omega\over r}=0.
$$
Il est alors possible de montrer par une m\'ethode d'\'energie que pour tout $p\in[1,\infty]$
(resp. $p\in]1,2]$) la norme de $\omega/r$ (resp. $r^{-1}\partial_z\omega$) dans $L^p$
(resp. $L^2_t(L^p)$) est contr\^ol\'ee par celle de ${\omega_0}/r.$
D'apr\`es la loi de Biot-Savart, on d\'emontre (voir Proposition \ref{Biot}) que
$$
|{u^r\over r}|
\lesssim
{1\over |\cdot|}\star|r^{-1}\partial_{z}\omega|.
$$
Ainsi la condition d'incompressibilit\'e nous permet de contr\^oler 
$\partial_ru^r$ puisque $\partial_ru^r=-{u^r\over r}-\partial_zu^z.$
Notre r\'esultat principal est le suivant (concernant la
d\'efinition de l'espace de Lorentz voir la section suivante):
\begin{theo}\label{existence_globale}
Soit $\omega_0\in L^{{3\over 2},1}(\R^3)$ tel que
${\omega_0\over r}\in L^{{3\over 2},1}(\R^3).$ Soit $u_0$ le champ de
vecteurs avec ${\mathop{\rm div}}\,u_0=0$ et 
$\omega_0=\nabla\times u_0$
donn\'e par la loi de Biot-Savart~:
$$
u_0(X)={1\over4\pi}\int_{\R^3}\frac{(X-Y)\times\omega_0(Y)}
{\vert X-Y\vert^3}\,dY.
$$
Alors le syst\`eme ${\rm(NS_v)}$ admet une solution globale $u$ tel que la la vorticit\'e $\omega$ satisfait
$$
\begin{aligned}
&\omega\in L^\infty_{loc}\big(\R_+;\,L^{{3\over 2},1}(\R^3)\big),
\hspace{1cm}
\partial_z\omega\in L^2_{loc}\big(\R_+;\,L^{{3\over 2},1}(\R^3)\big)
\\&
{\omega\over r}\in L^\infty_{loc}\big(\R_+;\,L^{{3\over 2},1}(\R^3)\big),\hspace{0,9cm}
\partial_z{\omega\over r}\in L^2_{loc}\big(\R_+;\,L^{{3\over 2},1}(\R^3)\big).
\end{aligned}
$$
De plus pour tout $t\geq 0,$ on a
$$
\|\omega(t)\|_{L^{{3\over 2},1}}
+\|\partial_z\omega\|_{L^2_t(L^{{3\over 2},1})}
\leq
C\|\omega_0\|_{L^{{3\over 2},1}}
\exp\big(Ct^{1\over2}\|r^{-1}\omega_0\|_{L^{{3\over 2},1}}\big)
$$
et
$$
\|r^{-1}\omega(t)\|_{L^{{3\over 2},1}}
+\|r^{-1}\partial_z\omega\|_{L^2_t(L^{{3\over 2},1})}
\leq
C\|r^{-1}\omega_0\|_{L^{{3\over 2},1}}.
$$
En outre, cette solution est unique si de plus 
$\partial_r\omega_0\in L^{{3\over2},1}.$ 
\end{theo}
\begin{remarque}\label{comparaison}
Rappelons que  pour des donn\'ees initiales de type Yudovitch 
R. Danchin \cite{Danchin} \`a d\'emontre que le syst\`eme d'Euler axisym\'etrique est globalement bien pose. Plus exactement il 
d\'emontre que le syst\`eme est globalement bien pos\'e lorsque 
$\omega_0\in L^{3,1}\cap L^\infty$ et $\omega_0/r\in L^{3,1}.$ 
R\'ecemment  H. Abidi et {\it al.} \cite{AHK} ont montr\'e que le 
syst\`eme d'Euler axisym\'etrique est globalement bien pose dans des espace critiques plus pr\'ecis\'ement lorsque 
$u_0\in B^{{3\over p}+1}_{p,1}$ pour 
$p\in[1,\infty]$ et $\omega_0/r\in L^{3,1}.$
\end{remarque}
\begin{remarque}
On note aussi quand obtient un  r\'esultat similaire que H. Abidi \cite{A}. En effet, dans cet article, l'auteur d\'emontre que le syst\`eme de Navier-Stokes axisym\'etrique 
(i.e, $\nu_h=\nu_v>0$) est globalement bien pos\'e lorsque la 
donn\'ee initiale v\'erifie  $u_0\in W^{2,p}(\R^3)$ pour $1<p<2$. 
\end{remarque}
Nous pouvons obtenir l'existence des solutions pour des donn\'ees initiales de r\'egularit\'e encore plus faible. L'unicit\'e en revanche semble \^etre beaucoup plus difficile \`a obtenir avec cette r\'egularit\'e tr\`es faible. Nous avons le r\'esultat suivant.
\begin{theo}\label{th2}
Soit $\omega_0\in L^{\frac 65}\cap L^{{6\over 5}+,1}(\R^3)$ tel que
${\omega_0\over r}\in L^{\frac 65}\cap L^{{6\over 5}+,1}(\R^3).$ Soit $u_0$ le champ de
vecteurs avec ${\mathop{\rm div}}\,u_0=0$ et 
$\omega_0=\nabla\times u_0$
donn\'e par la loi de Biot-Savart. Alors le syst\`eme ${\rm(NS_v)}$ admet une solution globale $u$ tel que
la la vorticit\'e $\omega$ satisfait
$$
\begin{aligned}
\big(\omega,\frac{\omega}{r}\big)\in L^\infty_{loc}
\big(\R_+;\,L^{\frac 65}\cap L^{{6\over 5}+,1}(\R^3)\big),
\hspace{1cm}
\big(\partial_z\omega, \partial_z\frac{\omega}{r}\big)\in 
L^2_{loc}\big(\R_+;\,L^{\frac 65}\cap L^{{6\over 5}+,1}(\R^3)\big).
\end{aligned}
$$
\end{theo}

\section{Notation et pr\'eliminaires}
On dit que $A\lesssim B$ s'il existe une constante $C$
strictement positive telle que $A\leq CB.$ La notation $C$ d\'esigne une constante g\'en\'erique qui peut changer d'une ligne \`a une autre.
Soient $X$ un espace de Banach et $p\in[1,\infty],$ on d\'esigne par $L^p(0,T;\,X)$ l'ensemble des
fonctions $f$ mesurables sur $(0,T)$ \`a valeurs dans $X,$ telles
que $t\longmapsto\Vert f(t)\Vert_X$ appartient \`a $L^p(0,T).$
On note $C([0,T);\,X)$ l'espace des fonctions continues de $[0,T)$ \`a
valeurs dans $X,$ $C_b([0,T);\, X)\overset{d\acute{e}f}{=}C([0,T);\,X)\cap L^{\infty}(0,T;\,X).$
Enfin on d\'esigne par $p'$ l'exposant conjugu\'e de $p$ d\'efini par
$\frac{1}{p}+\frac{1}{p'}=1.$

Avant d'introduire la d\'efinition de l'espace de Lorentz, on commence par rappel la r\'earrangement d'une fonction. Soit $f$ une fonction mesurable, on d\'efinit son
r\'earrangement $f^{*}:\R_+\to \R_+$ par la formule
$$
f^{*}(\lambda):=\inf\Big\{s\geq0;\,\big|\{x/\,|f(x)|>s\}\big|\leq\lambda\Big\}.
$$
\begin{definition} (espace de Lorentz)\label{espace_lorentz}
Soient $f$ une fonction mesurable et
$1\leq p,q\leq\infty.$
Alors $f$ appartient a l'espace de Lorentz $L^{p,q}$ si
\begin{displaymath}
\|f\|_{L^{p,q}}\overset{d\acute{e}f}{=}\begin{cases}
            \Big( \int^\infty_0(t^{1\over p}f^*(t))^q{dt\over t}\Big)
            ^{1\over q}<\infty&
\text{si $q<\infty$}\\
             \displaystyle\sup_{t>0}t^{1\over p}f^*(t)<\infty &\text
             {si $q=\infty$}.
        \end{cases}
\end{displaymath}
\end{definition}
Nous pouvons \'egalement d\'efinir les espaces de Lorentz comme interpolation r\'eelle des espaces de Lebesgue~:
$$
L^{p,q}:= (L^{p_0},L^{p_1})_{(\theta,q)},
$$
avec $1\le p_0<p<p_1\le\infty,$ $0<\theta<1$ satisfait
${1\over p}={1-\theta\over p_0}+{\theta\over p_1}$ et $1\leq q\leq\infty,$ muni de la norme
$$
\|f\|_{L^{p,q}}:=
\Big(\int_0^\infty\big(t^{-\theta}K(t,f)\big)^q
{dt\over t}\Big)^{1\over q}
$$
avec
$$
K(f,t):=\displaystyle\inf_{f=f_0+f_1}
\big\{\|f_0\|_{L^{p_0}}+t\|f_1\|_{L^{p_1}}\;\,\big|
\;f_0\in L^{p_0},\,f_1\in L^{p_1}\big\}.
$$

L'espace de Lorentz v\'erifie  les propri\'et\'es suivantes 
(pour plus de d\'etails voir \cite{ON})~:
\begin{proposition}\label{Neil}
Soient $f\in L^{p_1,q_1},$ $g\in L^{p_2,q_2}$ et
$1\leq p,q,p_j,q_j\leq\infty,$ pour $1\leq j\leq2.$
\vspace{0,5cm}

\begin{itemize}

\item
Si ${1\over p}={1\over p_1}+{1\over p_2}$ et
${1\over q}={1\over q_1}+{1\over q_2},$ alors
$$
\|fg\|_{L^{p,q}}
\lesssim
\|f\|_{L^{p_1,q_1}}\|g\|_{L^{p_2,q_2}}.
$$

\item
Si $1<p<\infty,$ ${1\over p}+1={1\over p_1}+{1\over p_2}$ et
${1\over q}={1\over q_1}+{1\over q_2},$ alors
$$
\|f\ast g\|_{L^{p,q}}
\lesssim
\|f\|_{L^{p_1,q_1}}\|g\|_{L^{p_2,q_2}},
$$
pour $p=\infty,$ et ${1\over q_1}+{1\over q_2}=1,$ alors
$$
\|f\ast g\|_{L^\infty}
\lesssim
\|f\|_{L^{p_1,q_1}}\|g\|_{L^{p_2,q_2}}.
$$

\item
Pour $1\leq p\leq\infty$ et $1\leq q_1\leq q_2\leq\infty,$ on a
$$
L^{p,q_1}\hookrightarrow L^{p,q_2}
\hspace{1cm}\mbox{et}\hspace{1cm}L^{p,p}=L^p.
$$
\end{itemize}

\end{proposition}
Dans le rep\`ere cylindrique $\omega=\nabla\times u$ admet une seule composante port\'ee par $e_{\theta}$ et dans le rep\`ere cart\'esienne deux composantes:
$$
\omega=(\omega^1,\omega^2,0)
$$
avec $\omega^1=\partial_y u^3-\partial_z u^2$ et
$\omega^2=\partial_z u^1-\partial_x u^3,$ $u^j$ pour $1\leq j\leq 3$ 
les composantes de $u$ dans la base cart\'esienne et $(x,y,z)$ les variables dans cette base.
Le fait que $u^\theta=0,$ alors dans le rep\`ere cylindrique, on a:
$$
\begin{aligned}
&u\cdot\nabla=u^r\partial_r+u^z\partial_z,
\\&
{\mathop{\rm div}}\,u=\partial_r u^r+{u^r\over r}+\partial_z u^z
\\&
\hspace{-3cm}\mbox{et}\hspace{2,5cm} u^r=\omega^\theta=0
\hspace{0,5cm}\mbox{sur la droite}\hspace{0,5cm} r=0.
\end{aligned}
$$
Le dernier point on peut le d\'eduire du fait que
$
u^\theta=0:
$
en effet, comme
$$
u^\theta=u\cdot e_{\theta}
$$
ainsi
\begin{equation}\label{u^r}
-yu^1+xu^2=0.
\end{equation}
Et par suite $u^1=0$ (resp. $u^2=0$) sur le plan $x=0$
(resp. $y=0$). Pour $\omega^\theta,$ on utilise le fait que $\omega$ est port\'ee par $e_{\theta},$ ce qui implique
$$
x\omega^1+y\omega^2=0,
$$
et par suite $\omega^1$ (resp. $\omega^2$) est nulle sur le plan $x=0$ (resp. $y=0$).
D'o\`u le r\'esultat.
Rappelons que si $u$ est solution de $(NS_v),$ alors $\omega$ v\'erifie l'\'equation suivante
$$
\partial_t \omega +(u^r\partial_r+u^z\partial_z)\omega - \frac{u^r}{r}\omega
- \partial^2_z\omega=0,
$$
mais comme $u^\theta=0,$ alors
\begin{equation}\label{tourbillon}
\partial_t \omega +(u\cdot\nabla)\omega-\frac{u^r}{r}\omega
- \partial^2_z\omega=0.
\end{equation}
Autrement dit, dans le cas axisym\'etrique, $(NS_v)$ se ram\`ene
\`a un probl\`eme d'\'evolution bidimensionnel.
Rappelons qu'en dimension 2, $\omega=\partial_xu^2-\partial_yu^1,$
v\'erifie l'\'equation de transport-diffusion suivante~:
$$
\partial_t\omega+(u\cdot\nabla)\omega-\partial^2_z\omega=0.
$$
En dimension 3 dans le cas axisym\'etrique $\frac{\omega}{r}$
joue un r\^ole similaire puisque
\begin{equation}\label{transport}
\partial_t\frac{\omega}{r}+(u\cdot\nabla)\frac{\omega}{r}
-\partial_z^2\frac{\omega}{r}=0.
\end{equation}
\section{D\'emonstration du th\'eor\`eme \ref{existence_globale} }
\subsection{Estimations a priori}
D'apr\`es l'\'equation (\ref{transport}) et la loi de Biot-Savart, on peut contr\^oler des quantit\'es
tr\`es importantes, qui nous permet de d\'emontrer l'existence globale. Plus exactement, on a la proposition suivante.

\begin{proposition}\label{Biot}
Soient $(p,q,\lambda)\in[1,\infty]^3,$ alors on a les in\'egalit\'es suivantes~:
\vspace*{0,2cm}

\begin{itemize}
\item Si ${3\over2}\leq p<\infty$ tel que
${1\over q}={1\over3}+{1\over p},$ alors
$$
\begin{aligned}
&\|u\|_{L^{p,\lambda}}
\lesssim
\|\omega\|_{L^{q,\lambda}},
\qquad
\|{u^r\over r}\|_{L^{p,\lambda}}
\lesssim
\|{\omega\over r}\|_{L^{q,\lambda}},
\qquad
\|\partial_zu^r\|_{L^{p,\lambda}}
\lesssim
\|\partial_z\omega\|_{L^{q,\lambda}},
\\
\\&
\|\partial_zu^z\|_{L^{p,\lambda}}
\lesssim
\|\partial_z\omega\|_{L^{q,\lambda}}
\hspace{0,3cm}
\mbox{et}\hspace{0,3cm}
\|\partial_zu^z\|_{L^{p,\lambda}}+\|\partial_ru^z\|_{L^{p,\lambda}}
\lesssim
\|\partial_r\omega\|_{L^{q,\lambda}}
+\|{\omega\over r}\|_{L^{q,\lambda}}.
\end{aligned}
$$
\vspace{0,1cm}
\item Si $3\leq p<\infty$ tel que
${1\over q}={2\over3}+{1\over p},$ alors
$$
\begin{aligned}
&\|u^r\|_{L^{p,\lambda}}
\lesssim
\|\partial_z\omega\|_{L^{q,\lambda}},
\hspace{1cm}
\|{u^r\over r}\|_{L^{p,\lambda}}
\lesssim
\|\partial_z{\omega\over r}\|_{L^{q,\lambda}}
\\
\\&
\|u^z\|_{L^{p,\lambda}}
\lesssim
\|\partial_r\omega\|_{L^{q,\lambda}}
+\|{\omega\over r}\|_{L^{q,\lambda}},
\hspace{0,5cm}
\|\partial_zu^z\|_{L^{p,\lambda}}
\lesssim
\|\partial_z\partial_r\omega\|_{L^{q,\lambda}}
+\|\partial_z{\omega\over r}\|_{L^{q,\lambda}}
\end{aligned}
$$
et
$$
\|\partial_ru^r\|_{L^{p,\lambda}}
\lesssim
\|\partial_z\partial_r\omega\|_{L^{q,\lambda}}
+\|\partial_z{\omega\over r}\|_{L^{q,\lambda}}.
$$

\vspace{0,5cm}

\item
Dans le cas limite, c'est-\`a-dire, $p=\infty$
$$
\begin{aligned}
&\|u\|_{L^\infty}
\lesssim
\|\omega\|_{L^{3,1}},
\qquad
\|u^r\|_{L^{\infty}}
\lesssim
\|\partial_z\omega\|_{L^{{3\over2},1}},
\qquad
\|{u^r\over r}\|_{L^{\infty}}
\lesssim
\|\partial_z{\omega\over r}\|_{L^{{3\over2},1}}
\\
\\&
\|u^z\|_{L^{\infty}}
\lesssim
\|\partial_r\omega\|_{L^{{3\over2},1}}
+\|{\omega\over r}\|_{L^{{3\over2},1}},
\hspace{0,5cm}
\|\partial_zu^z\|_{L^\infty}
\lesssim
\|\partial_z\partial_r\omega\|_{L^{{3\over2},1}}
+\|\partial_z{\omega\over r}\|_{L^{{3\over2},1}}
\end{aligned}
$$
et
$$
\|\partial_ru^r\|_{L^\infty}
\lesssim
\|\partial_z\partial_r\omega\|_{L^{{3\over2},1}}
+\|\partial_z{\omega\over r}\|_{L^{{3\over2},1}}.
$$
\end{itemize}
\end{proposition}

\begin{proof}
D'apr\`es la loi de Biot-Savart, on a
\begin{equation}\label{vitesse}
u(X)={1\over{4\pi}}\int_{\R^3}{{X-X'}\over{|X-X'|^3}}\times\,\omega(X')dX',
\end{equation}
avec $X=(x,y,z)$ et $X'=(x',y',z'),$ et par suite
$$
|u|
\lesssim
{1\over |\cdot|^2}\star|\omega|,
$$
or par d\'efinition de l'espace de Lorentz 
(d\'efinition \ref{espace_lorentz}), on a
$$
{1\over |X|^2}\in L^{{3\over2},\infty}(\R^3)
$$
ainsi gr\^ace \`a la Proposition \ref{Neil}, on en d\'eduit
$$
\|u\|_{L^{p,\lambda}}
\lesssim
\|\omega\|_{L^{{3p\over 3+p},\lambda}}
\hspace{1cm}
\mbox{pour ${3\over2}\leq p<\infty$}
\hspace{1cm}
\mbox{et}
\hspace{1cm}
\|u\|_{L^\infty}
\lesssim
\|\omega\|_{L^{3,1}}.
$$
D'apr\`es l'\'egalit\'e (\ref{vitesse}), on a
$$
u^1(x)=-{1\over{4\pi}}\int_{\R^3}{{z-z'}\over{|X-X'|^3}}\,\omega^2(X')dX'
$$
et
$$
u^2={1\over{4\pi}}\int_{\R^3}{{z-z'}\over{|X-X'|^3}}\,\omega^1(X')dX'
$$
avec
$\omega^1(X')=-\sin\theta'\,\omega^\theta(X')$ et
$\omega^2(X')=\cos\theta'\,\omega^\theta(X').$
Ainsi
$$
\begin{aligned}
u^r(X)
&=\cos\theta\, u^1(X)+\sin\theta\, u^2(X)
\\&
={1\over{4\pi}}\int_{\R^3}{{z-z'}\over{|X-X'|^3}}
\big\{-\cos\theta\cos\theta'-\sin\theta\sin\theta'\big\}
\omega^\theta(X') dX'
\end{aligned}
$$
o\`u l'on d\'esigne par $(r,\theta,z)$ les variables dans le rep\`ere cylindrique, rappelons que dans ce rep\`ere 
$X=(r\cos\theta,r\sin\theta,z)$ et $X'=(r'\cos\theta',r'\sin\theta',z').$
Et par suite
$$
\begin{aligned}
u^r(X)&=-{1\over{4\pi}}\int_{\R^3}{{z-z'}\over{|X-X'|^3}}
\big\{\cos\theta\sin\theta'+\sin\theta\cos\theta'\big\}\omega^\theta(X') dX'
\\&
=-{1\over{4\pi}}\int_{\R^3}{{z-z'}\over{|X-X'|^3}}
\cos(\theta-\theta')\omega^\theta(r',z')
r'dr'd\theta'dz',
\end{aligned}
$$
or
$$
{{z-z'}\over{|X-X'|^3}}=\partial_{z'}{1\over|X-X'|},
$$
ainsi par int\'egration par parties, on trouve
$$
u^r(X)
={1\over{4\pi}}\int_{\R^3}{1\over{|X-X'|}}
\cos(\theta-\theta')\partial_{z'}\omega^\theta(r',z')
r'dr'd\theta'dz'.
$$
Mais comme $u^r$ ne d\'epend pas de $\theta$ (X=(r,0,z)), alors
\begin{equation}\label{forme1}
u^r(t,r,z)
={1\over{4\pi}}\int_{\R^3}{1\over{|X-X'|}}
\cos\theta'\partial_{z'}\omega^\theta(t,r',z')
r'dr'd\theta'dz',
\end{equation}
ce qui implique que
$$
|u^r|
\lesssim
{1\over |\cdot|}\star|\partial_{z'}\omega|.
$$
Or par d\'efinition de l'espace de Lorentz, on a
$$
{1\over |X|}\in L^{3,\infty}(\R^3)
$$
ainsi gr\^ace \`a la Proposition \ref{Neil}, on obtient l'in\'egalit\'e 
souhait\'ee.
Pour la deuxi\`eme in\'egalit\'e de la proposition, gr\^ace a l'\'egalit\'e
(\ref{forme1}), on a
$$
|\partial_zu^r|
\lesssim
{1\over |\cdot|^2}\star|\partial_{z'}\omega|,
$$
en cons\'equence la Proposition \ref{Neil}, donne l'in\'egalit\'e d\'esir\'ee.
Pour ${u^r\over r},$ on utilise l'identit\'e (\ref{forme1}) et
on suit les m\^emes calculs de \cite{SY}, on trouve
$$
\begin{aligned}
u^r(t,r,z)
&={1\over{4\pi}}\int_{\R_+\times[0,2\pi]\times\R}
{\cos\theta'\partial_{z'}\omega^\theta(t,r',z')\over
{(r^2+r'^2-2rr'\cos\theta'+(z-z')^2)^{1\over 2}}} \,
r'dr'd\theta'dz'
\\&
\\&
=
{1\over{4\pi}}\int_{\R_+\times[-{\pi\over2},{\pi\over2}]\times\R}
{\cos\theta'\partial_{z'}\omega^\theta(t,r',z')\over
{(r^2+r'^2-2rr'\cos\theta'+(z-z')^2)^{1\over 2}}} \,
r'dr'd\theta'dz'
\\&
\\&
+{1\over{4\pi}}\int_{\R_+\times[{\pi\over2},{3\pi\over2}]\times\R}
{\cos\theta'\partial_{z'}\omega^\theta(t,r',z')\over
{(r^2+r'^2-2rr'\cos\theta'+(z-z')^2)^{1\over 2}}} \,
r'dr'd\theta'dz'
\end{aligned}
$$
pour la deuxi\`eme partie, on effectue le changement de variable suivant
$\theta'\to \theta'+\pi,$ on aura
\begin{equation}\label{forme2}
\begin{aligned}
u^r(t,r,z)
&=
{1\over{4\pi}}\int_{\R_+}\int_{-{\pi\over2}}^{{\pi\over2}}\int_{\R}
{\cos\theta'\partial_{z'}\omega^\theta(t,r',z')\over
{(r^2+r'^2-2rr'\cos\theta'+(z-z')^2)^{1\over 2}}} \,
r'dr'd\theta'dz'
\\&
\\&
-
{1\over{4\pi}}\int_{\R_+}\int_{-{\pi\over2}}^{{\pi\over2}}\int_{\R}
{\cos\theta'\partial_{z'}\omega^\theta(t,r',z')\over
{(r^2+r'^2+2rr'\cos\theta'+(z-z')^2)^{1\over 2}}} \,
r'dr'd\theta'dz'.
\end{aligned}
\end{equation}
Si $|X-X'|\leq r,$ on utilise l'\'egalit\'e (\ref{forme1}) et le fait que
$r'\leq 2r,$ on trouve
$$
\Big|\int_{|X-X'|\leq r}{\cos\theta'\partial_{z'}\omega^\theta(t,r',z')
\over |X-X'|}
r'dr'd\theta'dz'\Big|
\lesssim
r\int_{\R^3}{1\over{|X-X'|}}
\big|\partial_{z'}{\omega(t,X')\over r'}\big|dX'.
$$
Si $|X-X'|\geq r,$ on utilise l'\'egalit\'e (\ref{forme2}) et le fait que
$$
\begin{aligned}
\Big|\Big(r^2+r'^2+2rr'\cos\theta'+(z-z')^2\Big)^{-{1\over 2}}
&-
\Big(r^2+r'^2-2rr'\cos\theta'+(z-z')^2\Big)^{-{1\over 2}}\Big|
\\&
\leq
{2r\over|X-X'|^2},
\end{aligned}
$$
car $-{\pi\over2}\leq\theta'\leq{\pi\over2}.$ Ainsi dans cette r\'egion, on trouve
$$
\begin{aligned}
\Big|\int_{|X-X'|\geq r}{\cos\theta'\partial_{z'}\omega^\theta(t,r',z')
\over |X-X'|}
r'dr'd\theta'dz'\Big|
&\lesssim
r\int_{|X-X'|\geq r}{1\over{|X-X'|^2}}
|\partial_{z'}\omega(t,X')|dX'
\\&
\vspace*{2cm}
\lesssim
r\int_{|X-X'|\geq r}{r'\over{|X-X'|^2}}
\big|\partial_{z'}{\omega(t,X')\over r'}\big|dX',
\end{aligned}
$$
apr\`es, on utilise le fait que $r'=r'-r+r$ et $|r'-r|\leq|X-X'|,$ on obtient
$$
\Big|\int_{|X-X'|\geq r}{\cos\theta'\partial_{z'}\omega^\theta(t,r',z')
\over |X-X'|}r'dr'd\theta'dz'\Big|
\lesssim
r\int_{\R^3}{\over{|X-X'|}}
\big|\partial_{z'}{\omega(t,X')\over r'}\big|.
$$
Donc
$$
|u^r(t,X)|
\lesssim
r\int_{\R^3}{1\over{|X-X'|}}
\big|\partial_{z'}{\omega(t,X')\over r'}\big|dX',
$$
aussi on a
$$
|u^r(t,X)|
\lesssim
r\int_{\R^3}{1\over{|X-X'|^2}}
\big|{\omega(t,X')\over r'}\big|dX'.
$$
Pour conclure il suffit d'utiliser les lois de convolutions. 
Concernant $u^z$ d'apr\`es la loi de Biot-Savart, on a
\begin{equation}\label{u_z}
u^z(X)={1\over4\pi}\int_{\R^3}{\frac{(x-x')\omega^2(X')-(y-y')\omega^1(X')}
{|X-X'|^3}}\,dX'.
\end{equation}
Or
$$
{\frac{x-x'}{|X-X'|^3}}=\partial_{x'}{\frac{1}{|X-X'|}}
\quad\mbox{et}\quad -{\frac{y-y'}{|X-X'|^3}}=-\partial_{y'}{\frac{1}{|X-X'|}},
$$
alors par int\'egration par parties, on obtient
$$
u^z(X)={1\over4\pi}\int_{\R^3}{\frac{\partial_{y'}\omega^1-\partial_{x'}\omega^2}
{|X-X'|}}\,dX'.
$$
Mais en coordonn\'ees cylindriques, on a
$$
\begin{aligned}
&\partial_{x'}=\cos\theta'\partial_{r'}-{1\over r'}\sin\theta'\partial_{\theta'},
\quad
\partial_{y'}=\sin\theta'\partial_{r'}+{1\over r'}\cos\theta'\partial_{\theta'},
\\&
\omega^1=-\sin\theta'\omega^\theta
\quad\mbox{et}\quad
\omega^2=\cos\theta'\omega^\theta,
\end{aligned}
$$
et par suite
$$
\begin{aligned}
\partial_{y'}\omega^1-\partial_{x'}\omega^2
&=
-\sin^2\theta'\partial_{r'}\omega^\theta-{1\over r'}\cos^2\theta'\omega^\theta
-\big(\cos^2\theta'\partial_{r'}\omega^\theta+{1\over r'}\sin^2\theta'\omega^\theta)
\\&
=-\partial_{r'}\omega^\theta-{\omega^\theta\over r'},
\end{aligned}
$$
ainsi
\begin{equation}\label{uz}
u^z(X)=-{1\over4\pi}\int_{\R^3}{1\over |X-X'|}
\big(\partial_{r'}\omega^\theta+{\omega^\theta\over r'}\big)dX'.
\end{equation}
Donc
$$
|u^z|
\lesssim
{1\over |\cdot|}\star \big(|\partial_{r'}\omega|+|{\omega\over r'}|\big),
$$
de m\^eme pour la d\'eriv\'e par rapport \`a $z,$ on suit les m\^emes calculs et gr\^ace aux \'egalit\'es (\ref{u_z}) et (\ref{uz}), on trouve
$$
|\partial_zu^z|
\lesssim
\begin{cases}
{1\over |\cdot|^2}\star|\partial_{z'}\omega|\\
{1\over |\cdot|^2}\star\big(|\partial_{r'}\omega|
+|{\omega\over r'}|\big) \\
{1\over |\cdot|}\star\big(|\partial_{z'}\partial_{r'}\omega|+|\partial_{z'}{\omega\over r'}|\big).
\end{cases}
$$
Donc d'apr\`es les lois de convolutions, on d\'eduit les in\'egalit\'es souhait\'ees. Concernant
$\partial_ru^z,$ d'apr\`es l'\'egalit\'e (\ref{uz}), on a
$$
\partial_ru^z(X)={1\over4\pi}\int_{\R^3}{r-r'\cos\theta'\over |X-X'|^3}
\big(\partial_{r'}\omega^\theta-{\omega^\theta\over r'}\big)dX',
$$
alors
$$
|\partial_ru^z|
\lesssim
{1\over |\cdot|^2}\star\big(|\partial_{r'}\omega|+|{\omega\over r'}|\big)
$$
car
$$
\big|{r-r'\cos\theta'\over |X-X'|}\big|\le1.
$$
Enfin pour $\partial_ru^r,$ il suffit d'utiliser le fait que
$$
{\mathop{\rm div}}\,u=\partial_ru^r+{u^r\over r}+\partial_zu^z=0.
$$
D'o\`u la proposition.
\end{proof}

D'apr\`es la Proposition \ref{Biot}, on a besoin de contr\^ol\'e
$\omega$ dans l'espace de Lorentz $L^{{3\over 2},1},$ qui est l'objet de la proposition suivante. Plus exactement on va donner une estimation de la solution de l'\'equation transport-diffusion.

\begin{proposition}\label{Estimation_omega}
Soient $1<p<2,$ $1\leq q\leq\infty,$ $\omega_0\in L^{p,q}$
et $u$ un champ de vecteurs axisym\'etrique r\'eguli\`ere tels que
\mbox{${u^r\over r}\in L^1_t(L^\infty)$} et
${\mathop{\rm div}}\,u=0.$ Soit
$\omega\in L^\infty_t(L^{p,q})$ et
$\partial_z\omega\in L^2_t(L^{p,q})$
une solution du syst\`eme suivant
$$
{\rm(TD_{mod})}\;\left\{
\begin{array}{rl}
&\hspace{-0,5cm}\partial_t \omega +(u\cdot\nabla)\omega
-\frac{u^r}{r}\omega- \partial^2_z\omega=0\\
&\hspace{-0,5cm}\omega_{|t=0}=\omega_0.
\end{array}
\right.
$$
Alors
$$
\|\omega(t)\|_{L^{p,q}}+\|\partial_z\omega\|_{L^2_t(L^{p,q})}
\lesssim
\|\omega_0\|_{L^{p,q}}
e^{\int_0^t\|{u^r\over r}\|_{L^\infty}}.
$$
\end{proposition}
\begin{proof}
Tout d'abord on va estimer $\omega$ dans les espaces de Lebesgue.
Soit $1<p<\infty,$ on multiplie l'\'equation v\'erifi\'ee par $\omega$ par
$\vert\omega\vert^{p-1}{\rm sign}\,\omega.$
On obtient apr\`es int\'egrations par parties combin\'ees avec le fait que
${\mathop{\rm div}\nolimits\,u}= 0$
$$
\frac{1}{p}\frac{d}{dt}\Vert\omega\Vert_{L^p}^p
+{4(p-1)\over p^2}
\Big\Vert\partial_z\vert\omega\vert^{\frac{p}{2}}\Big\Vert_{L^2}^2
=\int_{\R^3}\frac{u^r}{r}|\omega|^pdx,
$$
par suite l'in\'egalit\'e de H\"older plus l'int\'egration par rapport au temps, impliquent
$$
\Vert\omega(t)\Vert_{L^p}^p
+{4(p-1)\over p}\Big\Vert\partial_z\vert\omega\vert^{\frac{p}{2}}
\Big\Vert_{L^2_t(L^2)}^2
\leq
\|\omega_0\|_{L^p}^p+p\int_0^t\|\frac{u^r}{r}(\tau)\|_{L^\infty}
\|\omega(\tau)\|_{L^p}^pd\tau.
$$
Ainsi le lemme de Gronwall, implique que
\begin{equation}\label{Prod-Scal}
\Vert\omega(t)\Vert_{L^p}^p
+{4(p-1)\over p}\Big\Vert\partial_z\vert\omega\vert^{\frac{p}{2}}
\Big\Vert_{L^2_t(L^2)}^2
\leq
\|\omega_0\|_{L^p}^p\exp\Big(p\int_0^t\|\frac{u^r}{r}(\tau)\|_{L^\infty}d\tau\Big).
\end{equation}
Pour estimer $\partial_z\omega$ dans $L^p$ nous allons utiliser
le lemme suivant. Admettons-le pour le moment.
\begin{lemme}\label{Lp}
Soient $1\leq p\leq2$ et $f\in L^p(\R^N)$ tel que
$\partial_i\vert u\vert^{\frac{p}{2}}\in L^2(\R^N).$ Alors
$$
\Vert \partial_if\Vert_{L^p}
\lesssim
\Big\Vert\partial_i\vert f\vert^{\frac{p}{2}}\Big\Vert_{L^2}
\Vert f\Vert_{L^p}^{\frac{2-p}{2}}.
$$
\end{lemme}
Pour $p\leq2,$ on en d\'eduit gr\^ace au Lemme \ref{Lp} et l'in\'egalit\'e
(\ref{Prod-Scal}), que
$$
\begin{aligned}
\Vert \partial_z\omega\Vert_{L^2_t(L^p)}
&\lesssim
\Big(\int_0^t\Big\Vert\partial_z\vert \omega\vert^{\frac{p}{2}}\Big\Vert_{L^2}^2
\Vert \omega\Vert_{L^p}^{2-p}d\tau\Big)^{1\over2}
\\&
\lesssim
\Vert \omega\Vert_{L^\infty_t(L^p)}^{2-p\over2}
\Big\Vert\partial_z\vert\omega\vert^{\frac{p}{2}}\Big\Vert_{L^2_t(L^2)}
\\&
\lesssim
\|\omega_0\|_{L^p}\exp\Big(\int_0^t\|\frac{u^r}{r}(\tau)\|_{L^\infty}d\tau\Big).
\end{aligned}
$$
Donc
\begin{equation}\label{omega}
\Vert\omega(t)\Vert_{L^p}+\Vert \partial_z\omega\Vert_{L^2_t(L^p)}
\lesssim
\|\omega_0\|_{L^p}\exp\Big(\int_0^t\|\frac{u^r}{r}(\tau)\|_{L^\infty}d\tau\Big).
\end{equation}
On d\'esigne par ${\mathcal{T}}$ et ${\mathcal{S}}$ les op\'erateurs suivants:
$$
\begin{aligned}
{\mathcal{T}}:\hspace{1cm} &L^p\longrightarrow L^p
\hspace{2cm}
{\mathcal{S}}:\hspace{1cm} L^p\longrightarrow L^2_t(L^p)
\\&
\omega_0\longmapsto \omega
\hspace{3,75cm}
\omega_0\longmapsto \partial_z\omega,
\end{aligned}
$$
avec $\omega$ solution du syst\`eme ${\rm(TD_{mod})}.$
Par d\'efinition, on a ${\mathcal{T}}$ et ${\mathcal{S}}$ sont lin\'eaires, alors par d\'efinition de l'espace de Lorentz (interpolation r\'eelle) et 
\cite{BER}, on obtient
\begin{equation}\label{Lorentz_omega}
\|\omega(t)\|_{L^{p,q}}+\|\partial_z\omega(\tau)\|_{L^2_t(L^{p,q})}
\lesssim
\|\omega_0\|_{L^{p,q}}\exp\Big(\int_0^t\|\frac{u^r}{r}(\tau)\|_{L^\infty}d\tau\Big).
\end{equation}
D'o\`u la proposition.
\end{proof}
En suivant les m\^emes calculs, on d\'eduit le corollaire suivant.
\begin{corollaire}\label{estimation_omega/r}
Soient $1<p<2,$ $1\leq q\leq\infty,$ $r^{-1}\omega_0\in L^{p,q}$
et $u$ un champ de vecteurs axisym\'etrique r\'eguli\`ere tel que
${\mathop{\rm div}}\,u=0.$ Soit
$r^{-1}\omega\in L^\infty_t(L^{p,q})$ et
$r^{-1}\partial_z\omega\in L^2_t(L^{p,q})$
une solution du syst\`eme suivant
$$
\left\{
\begin{array}{rl}
&\hspace{-0,5cm}\partial_t {\omega\over r} +(u\cdot\nabla){\omega\over r}-
\partial^2_z{\omega\over r}=0\\
&\hspace{-0,5cm}{\omega\over r}_{|t=0}={\omega_0\over r}.
\end{array}
\right.
$$
Alors
$$
\Big\|{\omega\over r}(t)\Big\|_{L^{p,q}}
+\Big\|\partial_z{\omega\over r}\Big\|_{L^2_t(L^{p,q})}
\lesssim
\Big\|{\omega_0\over r}\Big\|_{L^{p,q}}.
$$
\end{corollaire}
\begin{remarque}\label{p-qq}
D'apr\`es l'in\'egalit\'e (\ref{Prod-Scal}) et le fait que
$$
\|{\omega\over r}(t)\|_{L^p}
\le
\|{\omega_0\over r}\|_{L^p},
$$
on en d\'eduit gr\^ace \`a \cite{BER}, que
$\forall (p,q)\in]1,\infty[\times[1,\infty]$
$$
\|\omega(t)\|_{L^{p,q}}
\le
\|\omega_0\|_{L^{p,q}}e^{\int_0^t\|\frac{u^r}{r}(\tau)\|_{L^\infty}d\tau}
$$
et
$$
\|{\omega\over r}(t)\|_{L^{p,q}}
\le
\|{\omega_0\over r}\|_{L^{p,q}}.
$$
\end{remarque}
D'apr\`es la Proposition \ref{Biot}, le Corollaire \ref{estimation_omega/r}
et l'in\'egalit\'e de H\"older, on a
\begin{equation}\label{ur/r}
\begin{aligned}
\Big\|{u^r\over r}\Big\|_{L^1_t(L^\infty)}
\lesssim
\Big\|\partial_z{\omega\over r}\Big\|_{L^1_t(L^{{3\over2},1})}
&\lesssim
t^{1\over 2}\Big\|\partial_z{\omega\over r}\Big\|_{L^2_t(L^{{3\over2},1})}
\\&
\lesssim
t^{1\over 2}\Big\|{\omega_0\over r}\Big\|_{L^{{3\over 2},1}}.
\end{aligned}
\end{equation}
Et par suite pour tout $p\in]1,2[$ et $q\in[1,\infty]$ les in\'egalit\'es
(\ref{Lorentz_omega}) et (\ref{ur/r}), impliquent
\begin{equation}\label{effet_regularisant}
\|\omega(t)\|_{L^{p,q}}
+\|\partial_z\omega\|_{L^2_t(L^{p,q})}
\leq
C\|\omega_0\|_{L^{p,q}}
e^{Ct^{1\over 2}\|{\omega_0\over r}
\|_{L^{{3\over 2},1}}}.
\end{equation}
Ainsi la Proposition \ref{Biot}, Remarque \ref{p-qq} et l'in\'egalit\'e 
(\ref{ur/r}), impliquent que pour tout 
$(p,q)\in({3\over2},\infty)\times[1,\infty],$
$$
\|u(t)\|_{L^{p,q}}
\le
C\|\omega_0\|_{L^{{3p\over3+p},q}}
e^{Ct^{1\over 2}\|{\omega_0\over r}\|_{L^{{3\over 2},1}}}.
$$
Donc, si $\omega\in L^{{3\over2},1},$ alors l'in\'egalit\'e pr\'ec\'edente implique que $u\in L^{3,1},$ qui est inclus dans l'espace dual de 
$L^{{3\over2},1}.$ Et par suite
gr\^ace \`a la Proposition II.1 dans \cite {P_Lions} et l'\'equation que 
v\'erifie $\omega$ (\ref{tourbillon}), on d\'eduit le r\'esultat d'existence suivant.
\begin{corollaire}\label{existence}
Soit $\omega_0^\theta\in L^{{3\over 2},1}(\R^3)$ une fonction 
axisym\'etrique tel que
${\omega_0 ^\theta\over r}\in L^{{3\over 2},1}(\R^3).$ 
Soit $u_0$ le champ de vecteurs axisym\'etrique tel que 
${\mathop{\rm div}}\,u_0=0$ et avec vorticit\'e
$\omega_0=\omega_0^\theta(r,z)e_{\theta}$
donn\'e par la loi de Biot-Savart~:
$$
u_0(X)={1\over4\pi}\int_{\R^3}\frac{X-Y}{\vert X-Y\vert^3}\times\omega_0(Y)\,dY.
$$
Alors le syst\`eme ${\rm(NS_v)}$ admet une solution globale $u$ tel que la vorticit\'e $\omega$ satisfait
$$
\begin{aligned}
&\omega\in {\mathscr{C}}\big(\R_+;\,L^{{3\over 2},1}(\R^3)\big),
\hspace{1cm}
\partial_z\omega\in L^2_{loc}\big(\R_+;\,L^{{3\over 2},1}(\R^3)\big)
\\&
{\omega\over r}\in {\mathscr{C}}\big(\R_+;\,L^{{3\over 2},1}(\R^3)\big),\hspace{0,9cm}
\partial_z{\omega\over r}\in L^2_{loc}\big(\R_+;\,L^{{3\over 2},1}(\R^3)\big).
\end{aligned}
$$
De plus pour tout $t\geq 0,$ on a
$$
\|\omega(t)\|_{L^{{3\over 2},1}}
+\|\partial_z\omega\|_{L^2_t(L^{{3\over 2},1})}
\leq
C\|\omega_0\|_{L^{{3\over 2},1}}
e^{Ct^{1\over2}\|r^{-1}\omega_0\|_{L^{{3\over 2},1}}}
$$
et
$$
\|r^{-1}\omega(t)\|_{L^{{3\over 2},1}}
+\|r^{-1}\partial_z\omega\|_{L^2_t(L^{{3\over 2},1})}
\leq
C\|r^{-1}\omega_0\|_{L^{{3\over 2},1}}.
$$
\end{corollaire}
\noindent{\bf{D\'emonstration du Lemme \ref{Lp}}}.\\
Remarquons tout d'abord que
$$
\|\partial_if\|_{L^p}=\|\partial_i|f|\|_{L^p} \hspace{1cm}\mbox{et}\hspace{1cm}
|f|=|f|^{{p\over2}{2\over p}},
$$
ainsi, on a
$$
\partial_i|f|={p\over 2}\partial_i(|f|^{{p\over2}})|f|^{{2-p\over 2}}.
$$
Et par suite l'in\'egalit\'e de H\"older, implique que
$$
\Vert \partial_iu\Vert_{L^p}
\lesssim
\Big\Vert\partial_i\vert u\vert^{\frac{p}{2}}\Big\Vert_{L^2}
\Vert u\Vert_{L^p}^{2-p\over2}.
$$
D'o\`u le lemme. \hspace{10cm}$\square$
\subsection{Unicit\'e}
Pour d\'emontrer l'unicit\'e de solution pour le syst\`eme ${\rm(NS_v)},$
il suffit de le prouver pour l'\'equation (\ref{tourbillon}). Soient 
$\omega_1$ et $\omega_2$ deux solutions, et on d\'esignons par
$\delta\omega=\omega_2 -\omega_1$ leur diff\'erence, qui v\'erifie 
le syst\`eme suivant~:
$$
\left\{
\begin{array}{rl}
&\hspace{-0,5cm}\partial_t\delta\omega +(u_2\cdot \nabla)\delta\omega
-\partial_{z}^2\delta\omega=-(\delta u\cdot \nabla)\omega_1+{u^r_2\over r}\delta\omega
+{\delta u^r\over r}\omega_1\\
&\hspace{-0,5cm}{\delta\omega}_{|t=0}=0.
\end{array}
\right.
$$
L'espace dans lequel on va estimer la diff\'erence est $L^p$ avec
${6\over5}\le p<{3\over2}.$ Admettons  pour le moment le lemme suivant.
\begin{lemme}\label{difference}
Soient $\omega_i$ avec $1\le i\le2$ deux solutions de l'\'equation 
(\ref{tourbillon}) ayant les m\^emes donn\'ees initiales.
Supposons que pour $i=1,2$ on ait
$$
\omega_i\in L^{\infty}_t(L^{{3\over2},1}),
\hspace{0,5cm}
\partial_z\omega_i\in L^2_t(L^{{3\over2},1})
\hspace{0,5cm}\mbox{et}\hspace{0,5cm}
\partial_r\omega_i\in L^\infty_t(L^{{3\over2},1}).
$$
Alors
$$
\delta\omega\in L^\infty_t(L^p)
\hspace{1cm}
\mbox{et}
\hspace{1cm}
\partial_z|\delta\omega|^{p\over2}\in L^2_t(L^2).
$$
\end{lemme}
L'estimation d'\'energie implique que
$$
\begin{aligned}
{1\over p}{d\over dt}\|\delta\omega\|_{L^{p}}^{p}
+{4(p-1)\over p^2}
\Big\|\partial_z\vert\omega\vert^{\frac{p}{2}}\Big\|_{L^2}^2
&\le
\|{u^r_2\over r}\|_{L^\infty}\|\delta\omega\|_{L^{p}}^{p}
+\|{\omega_1\delta u^r\over r}\|_{L^{p}}\|\delta\omega\|_{L^{p}}^{p-1}
\\&
+\|(\delta u\cdot \nabla)\omega_1\|_{L^{p}}
\|\delta\omega\|_{L^{p}}^{p-1}.
\end{aligned}
$$
D'apr\`es l'in\'egalit\'e de H\"older, l'injection de Sobolev, la Proposition \ref{Biot} et le Lemme \ref{Lp}, on a
$$
\begin{aligned}
\|{\omega_1\delta u^r\over r}\|_{L^{p}}
&+\|(\delta u\cdot\nabla)\omega_1\|_{L^{p}}
\le
\big(\|{\omega_1\over r}\|_{L^{3\over2}}
+\|\partial_r\omega_1\|_{L^{3\over2}}\big)
\|\delta u^r\|_{L^{3p\over3-2p}}
\\&
+\|\partial_z\omega_1\|_{L^6_h(L^{\frac{3}{2}}_v)}
\|\delta u^z\|_{L^{6p\over6-p}_h(L^{\frac{3p}{3-2p}}_v)}
\\&
\lesssim
\big(\|{\omega_1\over r}\|_{L^{3\over2}}
+\|\partial_r\omega_1\|_{L^{3\over2}}\big)\|\partial_z\delta\omega\|_{L^{p}}
+\|\partial_z\partial_r\omega_1\|_{L^{\frac{3}{2}}}
\|\delta u^z\|_{L^{6p\over6-p}_h(L^{\frac{3p}{3-2p}}_v)}
\\&
\lesssim
\Big(\|{\omega_1\over r}\|_{L^{3\over2}}
+\|\partial_r\omega_1\|_{L^{3\over2}}\Big)
\|\partial_z\vert\delta\omega\vert^{p\over2}\|_{L^2}
\Vert\delta\omega\Vert_{L^{p}}^{2-p\over2}
+\|\partial_z\partial_r\omega_1\|_{L^{\frac{3}{2}}}
\|\delta u^z\|_{L^{6p\over6-p}_h(L^{\frac{3p}{3-2p}}_v)}.
\end{aligned}
$$
Concernant $\|\delta u^z\|_{L^{6p\over6-p}_h(L^{\frac{3p}{3-2p}}_v)}$
on utilise le fait que
$$
\Delta\delta u^z=\partial_r\delta\omega+\frac{\delta\omega}{r},
$$
et par suite par int\'egration par parties, on aura
$$
|\delta u^z|
\lesssim
\frac{1}{|\cdot|^2}\star|\delta\omega|,
$$
alors d'apr\`es les lois de convolution, on obtient
$$
\begin{aligned}
\|\delta u^z\|_{L^{6p\over6-p}_h(L^{\frac{3p}{3-2p}}_v)}
&\lesssim
\|\delta\omega\|_{L^{{6p\over6-p},{6p\over6-p}}_h(L^p_v)}
\\&
\lesssim
\|\delta\omega\|_{L^p}.
\end{aligned}
$$
Et par suite l'in\'egalit\'e de Young, implique que
$$
\begin{aligned}
{d\over dt}\|\delta\omega\|_{L^{p}}^{p}
&\le
\Big(\|{u^r_2\over r}\|_{L^\infty}+\|{\omega_1\over r}\|_{L^{3\over2}}^2
+\|\partial_r\omega_1\|_{L^{3\over2}}^2
+\|\partial_z\partial_r\omega_1\|_{L^{\frac{3}{2}}}\Big)
\|\delta\omega\|_{L^{p}}^{p}.
\end{aligned}
$$
Donc on a l'unicit\'e si  
$\partial_r\omega_1\in L^2_t(L^{3\over2})$ et 
$\|\partial_z\partial_r\omega_1\|_{L^{\frac{3}{2}}}$ puisque 
l'in\'egalit\'e (\ref{ur/r}) et le Corollaire \ref{estimation_omega/r}, impliquent
$\big(\|{u^r_2\over r}\|_{L^\infty}
+\|{\omega_1\over r}\|_{L^{3\over2}}^2)\in L^1_t.$

Dans un premi\`ere temps on d\'emontre que 
$\partial_r\omega_1\in L^2_t(L^{3\over2}).$ 
Plus exactement on prouve qu'on a propagation de la
r\'egularit\'e $\partial_r\omega$ dans l'espace de Lorentz 
$L^{{3\over2},1}$ plus l'effet r\'egularisant.
\subsection{Propagation de la r\'egularit\'e $\partial_r\omega$}
\begin{proposition}\label{partial}
Soient $\omega_0\in L^{{3\over2},1}\cap L^{3,1}$ tels que
$\omega_0/r\in L^{{3\over2},1}$ et
$\partial_r\omega_0\in L^{{3\over2},1}.$ Soit
$\partial_r\omega\in L^\infty_t(L^{{3\over2},1}),$
$\partial_z\partial_r\omega\in L^2_t(L^{{3\over2},1})$ une solution 
du syst\`eme suivant
$$
\left\{
\begin{array}{rl}
&\hspace{-0,5cm}\partial_t\partial_r\omega+(u\cdot\nabla)\partial_r\omega
-\partial_z^2\partial_r\omega
=-{u^r\over r}{\omega\over r}
+\partial_ru^r{\omega\over r}
+{u^r\over r}\partial_r\omega
-\partial_ru^r\partial_r\omega
-\partial_ru^z\partial_z\omega\\
&\hspace{-0,5cm}{\partial_r\omega}_{|t=0}=\partial_r\omega_0.
\end{array}
\right.
$$
Alors
$$
\|\partial_r\omega(t)\|_{L^{{3\over2},1}}+
\|\partial_z\partial_r\omega\|_{L^2_t(L^{{3\over2},1})}
\le
\Phi(t,\omega_0),
$$
avec
$$
\Phi(t,\omega_0)=e^{C\exp{\sqrt{t}C(\omega_0)}}
$$
\end{proposition}
\begin{proof}
En prenant le produit scalaire au sens $L^p$ pour $1<p\le2$ de
l'\'equation qui v\'erifie $\partial_r\omega$ combin\'es avec
$\partial_ru^r=-{u^r\over r}-\partial_zu^z$ et l'in\'egalit\'e de Hardy que implique que $\|r^{-1}\omega\|_{L^p}\lesssim\|\partial_r\omega\|_{L^p},$ on trouve
$$
\begin{aligned}
{1\over p}{d\over dt}\|\partial_r\omega\|_{L^p}^p&+{4(p-1)\over p^2}
\|\partial_z|\partial_r\omega|^{p\over2}\|_{L^2}^2
\le
4\|{u^r\over r}\|_{L^\infty}\|\partial_r\omega\|_{L^p}^p
+\|\partial_zu^z{\omega\over r}\|_{L^p}\|\partial_r\omega\|_{L^p}^{p-1}
\\&
+\|\partial_ru^z\partial_z\omega\|_{L^p}\|\partial_r\omega\|_{L^p}^{p-1}
+\int\partial_zu^z|\partial_r\omega|^p.
\end{aligned}
$$
Par int\'egration par parties plus l'in\'egalit\'e de Cauchy-Schwartz, on a
$$
\int\partial_zu^z|\partial_r\omega|^p=-2\int u^z
|\partial_r\omega|^{p\over2}\partial_z|\partial_r\omega|^{p\over2}
\le
2\|u^z\|_{L^\infty}\|\partial_z|\partial_r\omega|^{p\over2}\|_{L^2}
\|\partial_r\omega\|_{L^p}^{p\over2}.
$$
Gr\^ace \`a l'in\'egalit\'e de Young et le fait que
$\partial_ru^z=\partial_zu^r-\omega,$ on obtient
$$
\|\partial_ru^z\partial_z\omega\|_{L^p}=
\|\partial_zu^r\partial_z\omega-{1\over2}\partial_z\omega^2\|_{L^p}
\le
\|\partial_zu^r\partial_z\omega\|_{L^p}
+\|\partial_z\omega^2\|_{L^p}.
$$
Et par suite
\begin{equation}\label{inegalite1}
\begin{aligned}
{1\over p}{d\over dt}\|\partial_r\omega\|_{L^p}^p&+
\|\partial_z|\partial_r\omega|^{p\over2}\|_{L^2}^2
\lesssim
\Big(\|{u^r\over r}\|_{L^\infty}+\|u^z\|_{L^\infty}^2\Big)
\|\partial_r\omega\|_{L^p}^p
\\&
+\Big(\|\partial_zu^z{\omega\over r}\|_{L^p}
+\|\partial_zu^r\partial_z\omega\|_{L^p}
+\|\partial_z\omega^2\|_{L^p}\Big)
\|\partial_r\omega\|_{L^p}^{p-1}.
\end{aligned}
\end{equation}
gr\^ace \`a l'in\'egalit\'e de H\"older et par interpolation, on trouve
$$
\begin{aligned}
\|\partial_zu^z{\omega\over r}\|_{L^p}
&\le
\|{\omega\over r}\|_{L^p_h(L^\infty_v)}
\|\partial_zu^z\|_{L^\infty_h(L^p_v)}
\\&
\lesssim
\|{\omega\over r}\|_{L^p}^{\frac{p-1}{p}}
\|\partial_z{\omega\over r}\|_{L^p}^{\frac{1}{p}}
\|\partial_zu^z\|_{L^\infty_h(L^p_v)}.
\end{aligned}
$$
Comme $\Delta \partial_zu^z=\partial_z\partial_r\omega
+\partial_z\frac{\omega}{r},$ alors par integration par parties, on trouve
$$
\partial_zu^z=-{1\over{4\pi}}\int_{\R^3} \frac{r'-r\cos\theta'}{\big(r^2+{r'}^2-2rr'\cos\theta'+(z-z')^2\big)^{\frac{3}{2}}}\,
\partial_{z'}\omega\,r'dr'dz'd\theta',
$$
et par suite
$$
|\partial_zu^z|
\lesssim
\frac{1}{|X|^2}\star|\partial_z\omega|,
$$
ainsi
$$
\|\partial_zu^z\|_{L^\infty_h(L^p_v)}
\lesssim
\|\partial_z\omega\|_{L^{2,1}_h(L^p_v)}.
$$
Comme $1<p<2,$ alors par interpolation, on a
$$
\|f\|_{L^{2,1}(\R^2)}
\lesssim
\|f\|_{L^p}^{\frac{2p-2}{p}}\|\nabla f\|_{L^p}^{\frac{2-p}{p}}.
$$
Ainsi d'apr\`es l'in\'egalit\'e de Hardy, on trouve
$$
\|\partial_zu^z{\omega\over r}\|_{L^p}
\lesssim
\|{\omega\over r}\|_{L^p}^{\frac{p-1}{p}}
\|\partial_z{\omega\over r}\|_{L^p}^{\frac{1}{p}}
\|\partial_z\omega\|_{L^p}^{\frac{2p-2}{p}}
\|\partial_r\partial_z\omega\|_{L^p}^{\frac{2-p}{p}},
$$
et par suite le Lemme \ref{Lp}, implique que
$$
\|\partial_zu^z{\omega\over r}\|_{L^p}
\lesssim
\|{\omega\over r}\|_{L^p}^{\frac{p-1}{p}}
\|\partial_z{\omega\over r}\|_{L^p}^{\frac{1}{p}}
\|\partial_z\omega\|_{L^p}^{\frac{2p-2}{p}}
\|\partial_r\omega\|_{L^p}^{\frac{(2-p)^2}{2p}}
\|\partial_z|\partial_r\omega|^{\frac{p}{2}}\|_{L^p}^{\frac{2-p}{p}},
$$
ainsi on aura gr\^ace \`a l'in\'egalit\'e de Young
\begin{equation}\label{EST7}
\begin{aligned}
\|\partial_zu^z{\omega\over r}\|_{L^p}\|\partial_r\omega\|_{L^p}^{p-1}
&\leq
c_{\varepsilon}\|{\omega\over r}\|_{L^p}^{\frac{2p-2}{3p-2}}
\|\partial_z{\omega\over r}\|_{L^p}^{\frac{2}{3p-2}}
\|\partial_z\omega\|_{L^p}^{\frac{4p-4}{3p-2}}
\|\partial_r\omega\|_{L^p}^{\frac{3p^2-6p+4}{3p-2}}
\\&
+\varepsilon \|\partial_z|\partial_r\omega|^{\frac{p}{2}}\|_{L^p}^2.
\end{aligned}
\end{equation}
Mais comme
$
p-1\leq \frac{3p^2-6p+4}{3p-2}\leq p,
$
alors
$$
\begin{aligned}
{1\over p}{d\over dt}&\|\partial_r\omega\|_{L^p}^p+
\|\partial_z|\partial_r\omega|^{p\over2}\|_{L^2}^2
\\&
\lesssim
\Big(\|{\omega\over r}\|_{L^p}^{\frac{2p-2}{3p-2}}
\|\partial_z{\omega\over r}\|_{L^p}^{\frac{2}{3p-2}}
\|\partial_z\omega\|_{L^p}^{\frac{4p-4}{3p-2}}
+\|{u^r\over r}\|_{L^\infty}+\|u^z\|_{L^\infty}^2\Big)
\|\partial_r\omega\|_{L^p}^p
\\&
+\Big(\|{\omega\over r}\|_{L^p}^{\frac{2p-2}{3p-2}}
\|\partial_z{\omega\over r}\|_{L^p}^{\frac{2}{3p-2}}
\|\partial_z\omega\|_{L^p}^{\frac{4p-4}{3p-2}}
+\|\partial_zu^r\partial_z\omega\|_{L^p}
+\|\partial_z\omega^2\|_{L^p}\Big)
\|\partial_r\omega\|_{L^p}^{p-1}.
\end{aligned}
$$
Ainsi le lemme de Gronwall, implique que
$$
\begin{aligned}
&\|\partial_r\omega(t)\|_{L^p}+
\|\partial_z|\partial_r\omega|^{p\over2}\|_{L^2_t(L^2)}^{2\over p}
\\&
\le
\Big(\|\partial_r\omega_0\|_{L^p}
+C\int_0^t\big(\|{\omega\over r}\|_{L^p}^{\frac{2p-2}{3p-2}}
\|\partial_z{\omega\over r}\|_{L^p}^{\frac{2}{3p-2}}
\|\partial_z\omega\|_{L^p}^{\frac{4p-4}{3p-2}}+\|\partial_zu^r\partial_z\omega\|_{L^p}
+\|\partial_z\omega^2\|_{L^p}\big)d\tau\Big)
\\&
\times
e^{C\int_0^t\big(\|{\omega\over r}\|_{L^p}^{\frac{2p-2}{3p-2}}
\|\partial_z{\omega\over r}\|_{L^p}^{\frac{2}{3p-2}}
\|\partial_z\omega\|_{L^p}^{\frac{4p-4}{3p-2}}+\|{u^r\over r}\|_{L^\infty}
+\|u^z\|_{L^\infty}^2
+\|{\omega_0\over r}\|_{L^{3\over2}}^2\big)d\tau},
\end{aligned}
$$
enfin le Lemme \ref{Lp} et l'in\'egalit\'e, assurent que
$$
\begin{aligned}
&\|\partial_r\omega(t)\|_{L^p}+
\|\partial_z\partial_r\omega\|_{L^2_t(L^p)}
\le
\\&
C\Big(\|\partial_r\omega_0\|_{L^p}
+\int_0^t\big(\|{\omega\over r}\|_{L^p}^2
+\|\partial_z{\omega\over r}\|_{L^p}^2
+\|\partial_z\omega\|_{L^p}^2+\|\partial_zu^r\partial_z\omega\|_{L^p}
+\|\partial_z\omega^2\|_{L^p}\big)d\tau\Big)
\\&
\times
e^{C\int_0^t\big(\|{\omega\over r}\|_{L^p}^2
+\|\partial_z{\omega\over r}\|_{L^p}^2
+\|\partial_z\omega\|_{L^p}^2+\|{u^r\over r}\|_{L^\infty}
+\|u^z\|_{L^\infty}^2
+\|{\omega_0\over r}\|_{L^{3\over2}}^2\big)d\tau}.
\end{aligned}
$$
Par d\'efinition de l'espace de l'espace de Lorentz et d'apr\`es 
\cite{BER} l'estimation pr\'ec\'edente reste vraie dans 
$L^{{\frac{3}{2}},1},$ et par suite
$$
\begin{aligned}
&\|\partial_r\omega(t)\|_{L^{{\frac{3}{2}},1}}+
\|\partial_z\partial_r\omega\|_{L^2_t(L^{{\frac{3}{2}},1})}
\le
\\&
C\Big(\|\partial_r\omega_0\|_{L^{{\frac{3}{2}},1}}
+\int_0^t\big(\|{\omega\over r}\|_{L^{{\frac{3}{2}},1}}^2
+\|\partial_z{\omega\over r}\|_{L^{{\frac{3}{2}},1}}^2
+\|\partial_z\omega\|_{L^{{\frac{3}{2}},1}}^2
+\|\partial_zu^r\partial_z\omega\|_{L^{{\frac{3}{2}},1}}
+\|\partial_z\omega^2\|_{L^{{\frac{3}{2}},1}}\big)d\tau\Big)
\\&
\times
e^{C\int_0^t\big(\|{\omega\over r}\|_{L^{{\frac{3}{2}},1}}^2
+\|\partial_z{\omega\over r}\|_{L^{{\frac{3}{2}},1}}^2
+\|\partial_z\omega\|_{L^{{\frac{3}{2}},1}}^2
+\|{u^r\over r}\|_{L^\infty}
+\|u^z\|_{L^\infty}^2
+\|{\omega_0\over r}\|_{L^{3\over2}}^2\big)d\tau}.
\end{aligned}
$$
Rappelons que
$$
\|\omega(t)\|_{L^{{3\over 2},1}}
+\|\partial_z\omega\|_{L^2_t(L^{{3\over 2},1})}
\leq
C\|\omega_0\|_{L^{{3\over 2},1}}
\exp\big(Ct^{1\over2}\|r^{-1}\omega_0\|_{L^{{3\over 2},1}}\big)
$$
et
$$
\|\frac{u^r}{r}\|_{L^1_t(L^\infty)}
\lesssim
\sqrt{t}\|\frac{\omega_0}{r}\|_{L^{{\frac{3}{2}},1}}
\qquad\mbox{et}\qquad
\|\frac{\omega}{r}\|_{L^{{\frac{3}{2}},1}}
+\|\partial_z\frac{\omega}{r}\|_{L^2_t(L^{{\frac{3}{2}},1})}
\lesssim
\|\frac{\omega_0}{r}\|_{L^{{\frac{3}{2}},1}}.
$$
Donc gr\^ace aux Propositions (\ref{Neil}) et (\ref{Biot}), on aura
$$
\begin{aligned}
\int_0^t\|\partial_zu^r\partial_z\omega\|_{L^{{\frac{3}{2}},1}}
&\le
\int_0^t\|\partial_zu^r\|_{L^{6,2}}\|\partial_z\omega\|_{L^{2,2}}
\\&
\lesssim
\int_0^t\|\partial_z\omega\|_{L^2}^2,
\end{aligned}
$$
et par suite  les in\'egalit\'es (\ref{omega}) et (\ref{ur/r}), impliquent
\begin{equation}\label{EST_1}
\int_0^t\|\partial_zu^r\partial_z\omega\|_{L^{{\frac{3}{2}},1}}
\lesssim
\|\omega_0\|_{L^2}^2
e^{Ct^{1\over2}\|{\omega_0\over r}\|_{L^{{3\over2},1}}}.
\end{equation}
Concernant $\|\partial_z\omega^2\|_{L^1_t(L^{{3\over2},1})}$ 
la Proposition \ref{Neil}, implique que
$$
\begin{aligned}
\int_0^t\|\partial_z\omega^2\|_{L^{{3\over2},1}}
&\lesssim
\int_0^t\||\omega|^{1\over2}\|_{L^{6,2}}\|\partial_z|\omega|^{3\over2}
\|_{L^2}
\\&
\lesssim
t^{1\over2}\|\omega\|_{L^\infty_t(L^{3,1})}^{{1\over2}}
\|\partial_z|\omega|^{3\over2}\|_{L^2_t(L^2)},
\end{aligned}
$$
ainsi les in\'egalit\'es (\ref{Prod-Scal}) et (\ref{ur/r}) et la Remarque 
\ref{p-qq}, impliquent
\begin{equation}\label{EST_8}
\begin{aligned}
\int_0^t\|\partial_z\omega^2\|_{L^{{3\over2},1}}
\lesssim
t^{1\over2}\|\omega_0\|_{L^{3,1}}^2
e^{Ct^{1\over2}\|\omega_0/r\|_{L^{{3\over2},1}}}.
\end{aligned}
\end{equation}
Pour $\|u^z\|_{L^2_t(L^\infty)}^2$ la Proposition \ref{Biot} et la Remarque \ref{p-qq} entra\^inent
\begin{equation}\label{EST_4}
\int_0^t\|u^z\|_{L^\infty}^2
\lesssim
\int_0^t\|\omega\|_{L^{3,1}}^2
\lesssim
t\|\omega_0\|_{L^{3,1}}^2e^{Ct^{1\over2}\|{\omega_0\over r}\|_{L^{{3\over2},1}}}.
\end{equation}
D'o\`u la proposition.
\end{proof}
En fait on peut travailler avec des donn\'ees moins r\'eguli\`eres mais
le prix \`a payer est l'absence d'un contr\^ole explicite de la solution en fonction de la donn\'ee initiale. Ceci est pr{\'e}cis{\'e}
dans la proposition suivante.
\begin{proposition}\label{partial_r}
Soient $\omega_0\in L^{{3\over2},1}$ tels que
$\omega_0/r\in L^{{3\over2},1}$ et
$\partial_r\omega_0\in L^{{3\over2},1}.$ Soit
$\partial_r\omega\in L^\infty_t(L^{{3\over2},1}),$
$\partial_z\partial_r\omega\in L^2_t(L^{{3\over2},1})$ une solution du syst\`eme suivant
$$
\left\{
\begin{array}{rl}
&\hspace{-0,5cm}\partial_t\partial_r\omega+(u\cdot\nabla)\partial_r\omega
-\partial_z^2\partial_r\omega
=-{u^r\over r}{\omega\over r}
+\partial_ru^r{\omega\over r}
+{u^r\over r}\partial_r\omega
-\partial_ru^r\partial_r\omega
-\partial_ru^z\partial_z\omega\\
&\hspace{-0,5cm}{\partial_r\omega}_{|t=0}=\partial_r\omega_0.
\end{array}
\right.
$$
Alors
$$
\|\partial_r\omega(t)\|_{L^{{3\over2},1}}+
\|\partial_z\partial_r\omega\|_{L^2_t(L^{{3\over2},1})}
\le
\gamma(t,\omega_0).
$$
\end{proposition}
\begin{proof}
La preuve s'effectue en deux \'etapes, on d\'emontre premi\`erement qu'on a propagation de la r\'egularit\'e localement  et apr\`es on en 
d\'eduit globalement gr\^ace a l'effet r\'egularisant et la Proposition 
\ref{partial}.
En prenant le produit scalaire au sens $L^p$ pour $1<p\le2$ de
l'\'equation qui v\'erifie $\partial_r\omega$ combin\'es avec
$\partial_ru^r=-{u^r\over r}-\partial_zu^z$ et l'in\'egalit\'e de Hardy que implique
que $\|r^{-1}\omega\|_{L^p}\lesssim\|\partial_r\omega\|_{L^p},$ on trouve
\begin{equation}\label{Majo1}
\begin{aligned}
{1\over p}{d\over dt}\|\partial_r\omega\|_{L^p}^p+&{4(p-1)\over p^2}
\|\partial_z|\partial_r\omega|^{p\over2}\|_{L^2}^2
\le
4\|{u^r\over r}\|_{L^\infty}\|\partial_r\omega\|_{L^p}^p
\\&
+\int \partial_zu^z{\omega\over r}|\partial_r\omega|^{p-1}
+\int \partial_ru^z\partial_z\omega\partial_r\omega^{p-1}
+\int\partial_zu^z|\partial_r\omega|^p
\\&
\lesssim
\|{u^r\over r}\|_{L^\infty}\|\partial_r\omega\|_{L^p}^p
+\|\partial_zu^z{\omega\over r}\|_{L^p}\|\partial_r\omega\|_{L^p}^{p-1}
\\&
+\int \partial_ru^z\partial_z\omega|\partial_r\omega|^{p-1}
+\int\partial_zu^z|\partial_r\omega|^p.
\end{aligned}
\end{equation}
Rappelons que d'apr\`es l'in\'egalit\'e (\ref{EST7}), on a
$$
\begin{aligned}
\|\partial_zu^z{\omega\over r}\|_{L^p}\|\partial_r\omega\|_{L^p}^{p-1}
&\leq
c_{\varepsilon}\|{\omega\over r}\|_{L^p}^{\frac{2p-2}{3p-2}}
\|\partial_z{\omega\over r}\|_{L^p}^{\frac{2}{3p-2}}
\|\partial_z\omega\|_{L^p}^{\frac{4p-4}{3p-2}}
\|\partial_r\omega\|_{L^p}^{\frac{3p^2-6p+4}{3p-2}}
\\&
+\varepsilon \|\partial_z|\partial_r\omega|^{\frac{p}{2}}\|_{L^p}^2.
\end{aligned}
$$
Concernant Le terme
$$
\int\partial_r u^z \partial_z\omega |\partial_r\omega|^{p-1},
$$
on utilise le fait que $\partial_r u^z=\partial_z u^r-\omega$ et
 l'in\'egalit\'e de Minkowski, on obtient
$$
\begin{aligned}
\int\partial_r u^z \partial_z\omega &|\partial_r\omega|^{p-1}
\leq (\|\omega\partial_z\omega\|_{L^p}+\|\partial_z u^r\partial_z\omega\|_{L^p})\|\partial_r\omega\|_{L^p}^{p-1}
\\&
\leq
\big(\|\omega\|_{L^\infty_v(L^{2p}_h)}
+\|\partial_z u^r\|_{L^\infty_v (L^{2p}_h)}\big)
\|\partial_z\omega\|_{L^p_v(L^{2p}_h)}
\|\partial_r \omega\|_{L^p}^{p-1}
\\&
\leq
\big(\|\omega\|_{L^\infty_v(L^{2p}_h)}
+\|\partial_z u^r\|_{L^{2p}_h(L^\infty_v)}\big)
\|\partial_z\omega\|_{L^p_v(L^{2p}_h)}
\|\partial_r \omega\|_{L^p}^{p-1}.
\end{aligned}
$$
On se rappelle maintenant que $\partial_z\omega, \partial_z(\frac{\omega}{r})$ ainsi que $\partial_z\partial_r\omega$ sont dans $L^p.$ Donc, on a $\partial_z \omega\in L^p_v(W^{1,p}(\R^2_h))$, et comme $W^{1,p}(\R^2_h)\subset L^{p}(\R^2_h)\cap L^{\frac{2p}{2-p}}(\R^2_h)\subset L^{2p}(\R^2_h).$
Donc par interpolation et gr\^ace \`a l'in\'egalit\'e de Sobolev, on trouve
\begin{equation}\label{A}
\begin{aligned}
\|\partial_z\omega\|_{L^p_vL^{2p}_h}
&\lesssim
\|\partial_z\omega\|_{L^p}^{1-\frac{1}{p}}
\|\partial_z\omega\|_{L^p_v(L^{\frac{2p}{2-p}}_h)}^{\frac{1}{p}}
\\&
\lesssim
\|\partial_z\omega\|_{L^p}^{1-\frac{1}{p}}
\big(\|\partial_r\partial_z\omega\|_{L^p}
+\|\partial_z\frac{\omega}{r}\|_{L^p}\big)^{\frac{1}{p}}.
\end{aligned}
\end{equation}
D'autre part,
$$
\begin{aligned}
\|\omega\|_{L^\infty_v L^{2p}_h}
&\leq
\|\omega\|_{L^p_v(L^{2p}_h)}^{1-\frac{1}{p}}
\|\partial_z\omega\|_{L^p_vL^{2p}_h}^{1\over p}
\\&
\leq
\|\omega\|_{L^p}^{(1-\frac{1}{p})^2}
\big(\|\partial_r\omega\|_{L^p}
+\|\frac{\omega}{r}\|_{L^p}\big)^{\frac{1}{p}(1-\frac{1}{p})}
\\&
\times
\|\partial_z\omega\|_{L^p}^{\frac{1}{p}(1-\frac{1}{p})}
\big(\|\partial_z\partial_r\omega\|_{L^p}
+\|\partial_z\frac{\omega}{r}\|_{L^p}\big)^{\frac{1}{p^2}}.
\end{aligned}
$$
Concernant $\partial_z u^r,$ on utilise le fait que
\begin{equation}\label{C}
u^r=\frac{1}{|X|}\ast\partial_z\omega,
\end{equation}
on trouve
$$
\|\partial_z u^r\|_{L^\infty_v}
\lesssim
\frac{1}{{\sqrt{x^2+y^2}}^{1+\frac{1}{p}}}\ast\Vert\partial_z\omega\Vert_{L^p_v},
$$
ainsi pour $p\leq2,$ on trouve par interpolation
\begin{equation}\label{B}
\begin{aligned}
\|\partial_z u^r\|_{L^{2p}_h(L^\infty_v)}
&\lesssim
\Vert\partial_z\omega\Vert_{L^{2,1}_h(L^p_v)}
\\&
\lesssim
\Vert\partial_z\omega\Vert_{L^p}^{\frac{2(p-1)}{p}}
\Vert\nabla_h\partial_z\omega\Vert_{L^p}^{\frac{2-p}{p}}
\\&
\lesssim
\Vert\partial_z\omega\Vert_{L^p}^{\frac{2(p-1)}{p}}
\big(\|\partial_r\partial_z\omega\|_{L^p}
+\|\partial_z\frac{\omega}{r}\|_{L^p}\big)^{\frac{2-p}{p}}.
\end{aligned}
\end{equation}
En tenu compte de l'in\'egalit\'e de Hardy que implique que
$$
\|\frac{\omega}{r}\|_{L^p}\lesssim \|\partial_r\omega\|_{L^p}
\qquad\hbox{et}\qquad
\|\partial_z\frac{\omega}{r}\|_{L^p}\lesssim
\|\partial_z\partial_r\omega\|_{L^p}
\qquad\mbox{pour}\quad p\neq2,
$$
on obtient
$$
\begin{aligned}
\int\partial_r u^z \partial_z\omega |\partial_r\omega|^{p-1}
&\lesssim
\Big(\|\omega\|_{L^p}^{(1-\frac{1}{p})^2}
\|\partial_r\omega\|_{L^p}^{\frac{1}{p}(1-\frac{1}{p})}
\|\partial_z\omega\|_{L^p}^{\frac{1}{p}(1-\frac{1}{p})}
\|\partial_z\partial_r\omega\|_{L^p}^{\frac{1}{p^2}}
\\&
+
\Vert\partial_z\omega\Vert_{L^p}^{\frac{2(p-1)}{p}}
\|\partial_r\partial_z\omega\|_{L^p}^{\frac{2-p}{p}}\Big)
\|\partial_z\omega\|_{L^p}^{1-\frac{1}{p}}
\|\partial_r\partial_z\omega\|_{L^p}^{\frac{1}{p}}\|\partial_r\omega\|_{L^p}^{p-1}
\\&
\lesssim
\|\omega\|_{L^p}^{(1-\frac{1}{p})^2}\|\partial_z\omega\|_{L^p}^{1-\frac{1}{p^2}}
\|\partial_r\omega\|_{L^p}^{p-1+\frac{1}{p}(1-\frac{1}{p})}
\|\partial_r\partial_z\omega\|_{L^p}^{\frac{1}{p}(1+\frac{1}{p})}
\\&
+\Vert\partial_z\omega\Vert_{L^p}^{\frac{3(p-1)}{p}}
\|\partial_r\omega\|_{L^p}^{p-1}\|\partial_r\partial_z\omega\|_{L^p}^{\frac{3-p}{p}}
\end{aligned}
$$
gr\^ace au Lemme \ref{Lp} et l'in\'egalit\'e de Young, on aura
$$
\begin{aligned}
\|\partial_r\partial_z\omega\|_{L^p}^{\frac{1}{p}+\frac{1}{p^2}}
&\|\partial_r \omega\|_{L^p}^{p+\frac{1}{p}-\frac{1}{p^2}-1}
\|\omega\|_{L^p}^{(1-\frac{1}{p})^2}
\|\partial_z\omega\|_{L^p}^{1-\frac{1}{p^2}}
\\&
\lesssim
\|\partial_z|\partial_r\omega|^{\frac{p}{2}}
\|_{L^2}^{\frac{1}{p}+\frac{1}{p^2}}
\|\partial_r \omega\|_{L^p}^{p+\frac{3}{2p}-\frac{3}{2}}
\|\omega\|_{L^p}^{(1-\frac{1}{p})^2}
\|\partial_z\omega\|_{L^p}^{1-\frac{1}{p^2}}
\\&
\leq
\varepsilon\|\partial_z\vert\partial_r\omega\vert^{\frac{p}{2}}\|_{L^2}^2
+c_{\varepsilon}\|\omega\|_{L^p}^{2(p-1)\over{2p+1}}
\|\partial_z\omega\|_{L^p}^{{2(p+1)}\over{2p+1}}
\|\partial_r \omega\|_{L^p}^{p\,\frac{2p^2+3-3p}{2p^2-p-1}}
\end{aligned}
$$
et
$$
\begin{aligned}
\Vert\partial_z\omega\Vert_{L^p}^{\frac{3(p-1)}{p}}
\|\partial_r\omega\|_{L^p}^{p-1}\|\partial_r\partial_z\omega\|_{L^p}^{\frac{3-p}{p}}
\leq
\varepsilon\|\partial_z\vert\partial_r\omega\vert^{\frac{p}{2}}\|_{L^2}^2
+c_{\varepsilon}\Vert\partial_z\omega\Vert_{L^p}^2
\|\partial_r\omega\|_{L^p}^{\frac{3p^2-7p+6}{3p-3}}.
\end{aligned}
$$
Donc pour $1<p\leq2$
\begin{equation}\label{EST3}
\begin{aligned}
\int&\partial_r u^z \partial_z\omega |\partial_r\omega|^{p-1}
\leq
2\varepsilon\|\partial_z\vert\partial_r\omega\vert^{\frac{p}{2}}\|_{L^2}^2
\\&
+c_{\varepsilon}\|\omega\|_{L^p}^{2(p-1)\over{2p+1}}
\|\partial_z\omega\|_{L^p}^{{2(p+1)}\over{2p+1}}
\|\partial_r \omega\|_{L^p}^{p\,\frac{2p^2+3-3p}{2p^2-p-1}}
+c_{\varepsilon}\Vert\partial_z\omega\Vert_{L^p}^2
\|\partial_r\omega\|_{L^p}^{\frac{3p^2-7p+6}{3p-3}}.
\end{aligned}
\end{equation}
Enfin concernant le terme $\int\partial_zu^z|\partial_r\omega|^p,$ 
par int\'egration par parties plus l'in\'egalit\'e de Cauchy-Schwartz, on a
\begin{equation}\label{3/2}
\begin{aligned}
\int\partial_zu^z|\partial_r\omega|^p=&-2\int u^z
|\partial_r\omega|^{p\over2}\partial_z|\partial_r\omega|^{p\over2}.
\\&
\lesssim
\| u^z\|_{L^\infty}\|\partial_r\omega\Vert_{L^p}^{\frac{p}{2}}
\big\|\partial_z|\partial_r\omega|^{p\over2}\big\|_{L^2}.
\end{aligned}
\end{equation}
Comme $\Delta u^z=\partial_r\omega+\frac{\omega}{r},$ alors par integration par parties, on trouve
$$
u^z=-{1\over{4\pi}}\int_{\R^3} \frac{r'-r\cos\theta'}{\big(r^2+{r'}^2-2rr'\cos\theta'+(z-z')^2\big)^{\frac{3}{2}}}\,
\omega\,r'dr'dz'd\theta',
$$
et par suite
$$
|u^z|
\lesssim
\frac{1}{|X|^2}\star|\omega|,
$$
alors
$$
\|u^z\|_{L^\infty_h}
\lesssim
\frac{1}{|z|^{\frac{2-p}{p}}}\ast\|\omega\|_{L^{\frac{2p}{2-p}}_h},
$$
ainsi l'injection de Sobolev et l'in\'egalit\'e de Hardy, impliquent
$$
\begin{aligned}
\|u^z\|_{L^{\infty}_v(L^\infty_h)}
&\lesssim
\|\omega\|_{L^{{\frac{p}{p-1}},1}_v(L^{\frac{2p}{2-p}}_h)}
\\&
\lesssim
\|\partial_r\omega\|_{L^{{\frac{p}{p-1}},1}_v(L^p_h)}.
\end{aligned}
$$
Or par interpolation
$$
\|f\|_{L^{{\frac{p}{p-1}},1}(\R)}
\lesssim
\|f\|_{L^p}^{\frac{2p-2}{p}}\|\nabla f\|_{L^p}^{\frac{2-p}{p}},
$$
ainsi pour $1<p<2,$ on obtient gr\^ace au Lemme \ref{Lp}
$$
\begin{aligned}
\|u^z\|_{L^{\infty}}
&\lesssim
\|\partial_r\omega\|_{L^p}^{\frac{2p-2}{p}}
\|\partial_z\partial_r\omega\|_{L^p}^{\frac{2-p}{p}}
\\&
\lesssim
\|\partial_r\omega\|_{L^p}^{\frac{p}{2}}
\|\partial_z|\partial_r\omega|^{\frac{p}{2}}\|_{L^2}^{\frac{2-p}{p}}
\end{aligned}
$$
En injectant l'in\'egalit\'e pr\'ec\'edente dans l'in\'egalit\'e (\ref{3/2}) et on utilisant l'in\'egalit\'e de Young, on trouve pour $1<p<2$
\begin{equation}\label{EST4}
\int\partial_zu^z|\partial_r\omega|^p
\leq
c_{\varepsilon}\Vert \partial_r\omega\Vert_{L^p}^{\frac{p^2}{p-1}}
+\varepsilon\big\|\partial_z|\partial_r\omega|^{p\over2}\big\|_{L^2}^2.
\end{equation}
Donc les in\'egalit\'es (\ref{Majo1}), (\ref{EST7}), (\ref{EST3}) et 
(\ref{EST4}), impliquent
\begin{equation}\label{conclusion}
\begin{aligned}
{d\over dt}\|\partial_r\omega\|_{L^p}^p+&
\Big\|\partial_z|\partial_r\omega|^{p\over2}\Big\|_{L^2}^2
\lesssim
\|{u^r\over r}\|_{L^\infty}\|\partial_r\omega\|_{L^p}^p
+\Vert\partial_z\omega\Vert_{L^p}^2
\|\partial_r\omega\|_{L^p}^{\frac{3p^2-7p+6}{3p-3}}
\\&
+\|{\omega\over r}\|_{L^p}^{\frac{2p-2}{3p-2}}
\Big\|\partial_z{\omega\over r}\Big\|_{L^p}^{\frac{2}{3p-2}}
\|\partial_z\omega\|_{L^p}^{\frac{4p-4}{3p-2}}
\|\partial_r\omega\|_{L^p}^{\frac{3p^2-4p+2}{3p-2}}
\\&
+\|\omega\|_{L^p}^{{2(p-1)}\over{2p+1}}
\|\partial_z\omega\|_{L^p}^{{2(p+1)}\over{2p+1}}
\|\partial_r \omega\|_{L^p}^{p\,\frac{2p^2+3-3p}{2p^2-p-1}}
+\Vert \partial_r\omega\Vert_{L^p}^{\frac{p^2}{p-1}}.
\end{aligned}
\end{equation}
En integrant l'in\'egalit\'e pr\'ec\'edente et on tenu compte des 
l'in\'galit\'es (\ref{ur/r}) et de Hardy, on obtient
$$
\begin{aligned}
\|\partial_r\omega\|_{L^\infty_t(L^p)}^p+&
\Big\|\partial_z|\partial_r\omega|^{p\over2}\Big\|_{L^2_t(L^2)}^2
\lesssim
\|\partial_r\omega_0\|_{L^p}^p
+t^{\frac{1}{2}}\|{\omega_0\over r}\|_{L^{{3\over2},1}}
\|\partial_r\omega\|_{L^\infty_t(L^p)}^p
\\&
+\Vert\partial_z\omega\Vert_{L^2_t(L^p)}^2
\|\partial_r\omega\|_{L^\infty_t(L^p)}^{\frac{3p^2-7p+6}{3p-3}}
\\&
+t^{\frac{p-1}{3p-2}}
\Big\|\partial_z{\omega\over r}\Big\|_{L^2_t(L^p)}^{\frac{2}{3p-2}}
\|\partial_z\omega\|_{L^2_t(L^p)}^{\frac{2p-2}{3p-2}}
\|\partial_r\omega\|_{L^\infty_t(L^p)}^{\frac{3p^2-2p}{3p-2}}
\\&
+t^{\frac{p}{2p+1}}\|\omega\|_{L^\infty_t(L^p)}^{{2(p-1)}\over{2p+1}}
\|\partial_z\omega\|_{L^2_t(L^p)}^{{2(p+1)}\over{2p+1}}
\|\partial_r \omega\|_{L^\infty_t(L^p)}^{p\,\frac{2p^2+3-3p}{2p^2-p-1}}
+t\Vert \partial_r\omega\Vert_{L^\infty_t(L^p)}^{\frac{p^2}{p-1}}.
\end{aligned}
$$
Et par suite pour $1<p<2,$ il existe $T>0$ tels que
$
\partial_r\omega\in L^\infty_T(L^p)$
et
$\partial_z\partial_r\omega\in L^2_T(L^p),$ ainsi  il existe $t_0\in[0,T]$ tels que
$\partial_r\omega(t_0)\in L^p$
et $\partial_z\partial_r\omega(t_0)\in L^2_T(L^p).$  Par d\'efinition de l'espace de Lorentz, on en d\'eduire les m\^emes r\'esultats. Ainsi il existe $t_1$ tel que
$\omega(t_1)\in L^{{3\over2},1}\cap L^{3,1}.$ Pour conclure la 
d\'emonstration il suffit d'utilis\'e la Proposition \ref{partial}. 
D'o\`u la proposition.
\end{proof}

\noindent{\bf{D\'emonstration du Lemme \ref{difference}}}.\\
Pour d\'emontrer que $\delta\omega\in L^\infty_t(L^p)$ et
$\partial_z|\omega|^{p\over2}\in L^2_t(L^2)$ il suffit de prouver que
$(u_2\cdot \nabla)\delta\omega+(\delta u\cdot \nabla)\omega_1
-{u^r_2\over r}\delta\omega-{\delta u^r\over r}\omega_1\in L^1_t(L^p)$ pour $p\leq{3\over2}.$ D'apr\`es l'in\'egalit\'e de H\"older, par interpolation et gr\^ace \`a la Proposition \ref{Biot} 
et \cite{TARTAR}, on a
$$
\begin{aligned}
\|(u_2\cdot \nabla)\delta\omega\|_{L^p}
&\le
\|u_2\|_{L^{3p\over3-2p}}\sum_{i=1}^2(\|\partial_r\omega_i\|_{L^{3\over2}}+
(\|\partial_z\omega_i\|_{L^{3\over2}})
\\&
\le
\|u_2\|_{L^3}^{3-2p\over p}\|u_2\|_{L^\infty}^{3(p-1)\over p}
\sum_{i=1}^2(\|\partial_r\omega_i\|_{L^{3\over2}}+\|\partial_z\omega_i\|_{L^{3\over2}})
\\&
\lesssim
\|\omega_2\|_{L^{3\over2}}^{3-2p\over p}\|\omega_2\|_{L^{3,1}}^{3(p-1)\over p}
\sum_{i=1}^2(\|\partial_r\omega_i\|_{L^{3\over2}}+\|\partial_z\omega_i\|_{L^{3\over2}})
\\&
\lesssim
\|\omega_2\|_{L^{3\over2}}^{3-2p\over p}
\big(\|\partial_r\omega_2\|_{L^{{3\over2},1}}+
\|{\omega_2\over r}\|_{L^{{3\over2},1}}+\|\partial_z\omega_2\|_{L^{{3\over2},1}}\big)^{3(p-1)\over p}
\\&
\times
\sum_{i=1}^2(\|\partial_r\omega_i\|_{L^{3\over2}}+\|\partial_z\omega_i\|_{L^{3\over2}})
\end{aligned}
$$
et par suite les deux propositions pr\'ec\'edentes combin\'ee avec le Corollaire \ref{existence}, impliquent 
$(u_2\cdot\nabla)\delta\omega\in L^1_t(L^p),$ les m\^emes calculs donnent $(\delta u\cdot \nabla)\omega_1\in L^1_t(L^p).$
Pour ${u^r_2\over r}\delta\omega$ gr\^ace \`a l'in\'egalit\'e de H\"older, par interpolation et la Proposition \ref{Biot}, on obtient
$$
\begin{aligned}
\|{u^r_2\over r}\delta\omega\|_{L^p}
\le
\|u^r_2\|_{L^{3p\over3-2p}}\|{\delta\omega\over r}\|_{L^{3\over2}}
\le
&\sum_{i=1}^2\|{\omega_i\over r}\|_{L^{3\over2}}\|u^r_2\|_{L^3}^{3-2p\over p}
\|u^r_2\|_{L^\infty}^{3(p-1)\over p}
\\&
\lesssim
\sum_{i=1}^2\|{\omega_i\over r}\|_{L^{3\over2}}\|\omega_2\|_{L^{3\over2}}^{3-2p\over p}
\|\partial_z\omega_2\|_{L^{{3\over2},1}}^{3(p-1)\over p}.
\end{aligned}
$$
Et par suite le Corollaire \ref{existence} et le fait que
${3(p-1)\over p}\le2$ impliquent ${u^r_2\over r}\delta\omega
\in L^1_t(L^p)$
les m\^emes calculs donnent ${\delta u^r\over r}\omega_1
\in L^1_t(L^p).$
D'o\`u le lemme. \hspace{1cm}$\square$

\section{Existence pour des donn\'ees moins 
r\'eguli\`eres}

Dans cette partie nous d\'emontrons le Th\'eor\`eme \ref{th2} d'existence des solutions pour des donn\'ees initiales moins 
r\'eguli\`eres. Pour cela nous avons besoin de prendre en compte encore plus des estimations anisotropes sur $\frac{u^r}{r}$. 
Nous avons, pour tout $1<p\leq \frac 32$,  l'in\'egalit\'e suivante
$$
\|\frac{u^r}{r}\|_{L^\infty_h(L^{\frac{p}{3-2p}}_v)}
\leq 
C\|\partial_z\frac{\omega}{r}\|_{L^{p,1}}.
$$
En effet: d'apr\`es les estimations de la Proposition \ref{Biot} on a
$$
|\frac{u^r}{r}|
\lesssim
\frac{1}{|X|}\star|\partial_z\frac{\omega}{r}|.
$$
Donc
$$
\|\frac{u^r}{r}\|_{L^\infty_h}
\lesssim
\|\frac{1}{\sqrt{|X_h|^2+z^2}}\|_{L^{p'}_h}\star
\|\partial_z\frac{\omega}{r}\|_{L^p_h}
$$
En utilisant le fait que la primitive de $r(r^2+z^2)^{-{p'\over2}}$ 
est $\sqrt{r^2+z^2}^{2-p'}$ a une constante pr\`es, on trouve
$$
\|\frac{u^r}{r}\|_{L^\infty_h}
\lesssim 
\frac{1}{|z|^{\frac{2}{p}-1}}\star\|\partial_z\frac{\omega}{r}\|_{L^p_h}.
$$
On prend maintenant la norme $L^{\frac{p}{3-2p}}$ en variable verticale pour obtenir
$$
\|\frac{u^r}{r}\|_{L^\infty_h(L^{\frac{p}{3-2p}}_v)}\leq C\|\partial_z\frac{\omega}{r}\|_{L^{p,1}}.
$$
On peut ainsi contr\^oler la norme de $\omega$ dans tout $L^p,$ 
rappelons que $\omega$ v\'erifie l'\'equation suivante
$$
\partial_t\omega+u\nabla\omega-\frac{u^r}{r} \omega-\partial_z^2\omega=0
$$
Donc pour $1<p\leq3/2$, on a
$$
\begin{aligned}
\frac12\frac{d}{dt}\||\omega(t)|^{p/2}\|_{L^2}^{2}
+\|\partial_z(|\omega|^{p/2})\|^2_{L^2}
&\leq
\int |\frac{u^r}{r}||\omega|^{p/2}|\omega|^{p/2}
\\&
\leq 
\|\frac{u^r}{r}\|_{L^\infty_hL^{\frac{p}{3-2p}}_v}
\||\omega|^{p/2}\|_{L^2_h(L^{\frac{2p}{3(p-1)}}_v)}^2.
\end{aligned}
$$
Comme $H^s(\R_v)\subset L^{\frac{2p}{3p-3}}(\R_v)$ pour 
$s=(3-2p)/(2p)$, alors
$$
\||\omega|^{p/2}\|_{L^2_h(L^{\frac{2p}{3p-3}}_v)}^2 
\leq 
\||\omega|^{p/2}\|_{L^2}^{(4p-3)/p}
\|\partial_z(|\omega|^{p/2})\|_{L^2}^{(3-2p)/p}.
$$
Donc
\begin{align*}
\frac12\frac{d}{dt}\||\omega(t)|^{p/2}\|_{L^2}^{2}
+\|\partial_z(|\omega|^{p/2})\|^2_{L^2}&
\leq 
\frac12\|\frac{u^r}{r}\|_{L^\infty_h(L^{\frac{p}{3-2p}}_v)}^{\frac{2p}{4p-3}}
\||\omega|^{p/2}\|^2_{L^2}
+\frac12\|\partial_z(|\omega|^{p/2})\|_{L^2}^2\\
&\leq 
C\|\partial_z\frac{\omega}{r}\|_{L^{p,1}}^{\frac{2p}{4p-3}}
\||\omega|^{p/2}\|^2_{L^2}
+\frac12\|\partial_z(|\omega|^{p/2})\|_{L^2}^2.
\end{align*}
Par le lemme de Gronwall et vu que 
$\|\partial_z\frac{\omega}{r}\|_{L^{\frac{2p}{4p-3}}_t(L^{p,1})}
\leq 
t^{\frac{3(p-1)}{4p-3}}\|\frac{\omega_0}{r}\|_{L^{p,1}},$ 
nous obtenons
$$
\|\omega\|_{L^p}
+\|\partial_z \omega\|_{L^2_t(L^p)}
\leq \|\omega_0\|_{L^p}
\exp(Ct^{\frac{3(p-1)}{4p-3}}\|\frac{\omega_0}{r}\|_{L^{p,1}}),
$$
et par interpolation
$$
\|\omega\|_{L^{p,1}}+\|\partial_z \omega\|_{L^2_t(L^{p,1})}
\leq 
\|\omega_0\|_{L^{p,1}}
\exp(Ct^{\frac{3(p-1)}{4p-3}}\|\frac{\omega_0}{r}\|_{L^{p,1}}).
$$

\

En particulier l'in\'egalit\'e pr\'ec\'edente est valable pour $p=6/5,$ ainsi on peut montrer l'existence globale d'une solution lorsque 
$\frac{\omega}{r}\in L^{\frac 65+,1}$ et 
$\omega_0\in L^{\frac 65+,1}$. Tout d'abord, on note que 
$\omega_0\in L^{\frac 65,1}$ implique que $u_0\in L^2$ et par l'estimation d'\'energie on a
$$
\|u(t)\|^2_{L^2}+2\int_0^t\|\partial_z u\|^2_{L^2}
\leq 
\|u_0\|^2_{L^2}.
$$
D'autre part, comme $\omega\in L^\infty_t(L^{\frac 65+,1})$ et
$\|\omega\|_{L^p}\approx\|\nabla u\|_{L^p}$ pour $1<p<+\infty,$ alors
$u\in L^\infty_t(\dot W^{1,\frac 65+})$ et donc finalement $u\in L^\infty_t(W^{1,\frac 65+}(\R^3))$ qui est un sous espace de $L^\infty_t(L^2(\R^3))$ avec l'inclusion compacte dans la topologie de 
$L^2_{loc}(\R^3)$ \`a $t$ fix\'e.  Donc, on peut construire la solution en utilisant uniquement  $\omega\in L^{\frac 65}\cap L^{\frac 65+,1}$ et 
$\frac{\omega}{r}\in L^{\frac 65}\cap L^{\frac 65+,1}$ par passage  \`a la limite dans 
une suite des solutions approch\'ees,  axisym\'etriques et  r\'eguli\`eres de l'\'equation
$$
\partial_t u+ \text{div\,}(u\otimes u)-\partial_z^2 u=-\nabla p.
$$
Plus pr\'ecisement, soit $u_0\in L^2(\R^3)$ de sorte que $\omega_0\in L^{\frac 65}\cap L^{\frac 65+,1}$ et $\frac{\omega_0}{r}\in L^{\frac 65}\cap L^{\frac 65+,1}$. Soit $J_n$ l'op\'erateur de troncature sur les basses fr\'equences d\'efini par $J_n u={\mathcal F}^{-1}(\chi(\xi 2^{-n}){\mathcal F}u(\xi))$, o\`u $\mathcal F$ d\'enote la transform\'ee de Fourier et $\chi$ est une fonction radiale r\'eguli\`ere qui vaut $1$ sur une boule autour de z\'ero. On sait que pour  $u_0$  champ axisym\'etrique sans swirl on a $J_n u_0$ est aussi axysim\'etrique sans swirl et r\'eguli\`er (voir \cite{AHK}). Donc,  il existe un unique solution globale  r\' eguli\`ere, axisym\'etrique et sans swirl $u^n$ solution du probl\`eme
\begin{equation*}(NS_{n})
\begin{cases}
\partial_t u_n+\text{div\,}(u_n\otimes u_n)-n^{-1}\Delta_h u_n-\partial_3^2 u_n=-\nabla p_n\\
\text{div\,} u_n=0\\
u_n|_{t=0}=J_n u_0.
\end{cases}
\end{equation*}
En tenant compte du fait que $J_n\omega_0$ et $\frac{J_n\omega_0}{r}$ sont uniform\'ement born\'es dans $L^{\frac 65}\cap L^{\frac 65+,1}$ (voir \cite{Danchin}) nous obtenons que $u_n$ est une suite uniform\'ement born\'ee dans $L^\infty_t(W^{1,\frac 65+}(\R^3))$. En utilisant l'\'equation v\'erifi\'ee par  $u_n$ on trouve ais\'ement que $\partial_t u_n$ est born\'ee dans $L^\infty_t(H^{-N})$ pour $N$ assez grand. En tenant compte du fait que l'inclusion $W^{1,\frac 65+}(\R^3)$ dans $L^2_{loc}(\R^3)$ est compacte et comme $u_n$ est born\'ee dans 
$C_{loc}(H^{-N})$ nous obtenons par le lemme de Arzela-Ascoli, quitte \`a extraire une sous suite, que $u_n$ converge fortement vers un $u$ dans $C_{loc}(H^{-N}_{loc})$. En interpolant avec le fait que $u_n$ est suite born\'ee dans $L^\infty(W^{1,\frac65 +})$ on trouve que 
$u_N\to u$ dans $L^\infty_{loc}(L^2(\R^3))$. Cela suffit pour passer 
\`a la limite dans les termes non-lin\'eaires et on trouve que 
$u_n\otimes u_n\to u\otimes u$ dans ${\mathcal D}'.$ 
Finalement, par passage \`a la limite dans 
$(NS_n)$ nous obtenons une solution globale axisym\'etrique sans swirl $u$ de $(NS_v)$.

\end{document}